\documentclass[letterpaper, 11pt,  reqno]{amsart}
\usepackage[margin=1.2in,marginparwidth=1.5cm, marginparsep=0.5cm]{geometry}

\usepackage{amsmath,amssymb,amscd,amsthm,amsxtra, esint, xcolor}

\usepackage[implicit=true]{hyperref}
\allowdisplaybreaks[2]

\usepackage{color}

\definecolor{gr}{rgb}   {0.,   0.69,   0.23 }
\definecolor{bl}{rgb}   {0.,   0.5,   1. }
\definecolor{mg}{rgb}   {0.85,  0.,    0.85}
\definecolor{yl}{rgb}   {0.8,  0.7,   0.}
\definecolor{or}{rgb}  {0.7,0.2,0.2}

\newtheorem{theorem}{Theorem} [section]

\newtheorem{lemma}[theorem]{Lemma}
\newtheorem{proposition}[theorem]{Proposition}
\newtheorem{remark}[theorem]{Remark}

\newtheorem{corollary}[theorem]{Corollary}

\DeclareMathOperator*{\intt}{\int}

\DeclareMathOperator*{\supp}{supp}

\newcommand{\noi}{\noindent}
\newcommand{\Z}{\mathbb{Z}}
\newcommand{\R}{\mathbb{R}}

\newcommand{\T}{\mathbb{T}}

\newcommand{\Tdl}{\mathcal{T}_{\dl} }
\newcommand{\Qdl}{\mathcal{Q}_{\dl} }

\newcommand{\mN}{\mathcal N}

\newcommand{\too}{\longrightarrow}

\let\P= \undefined
\newcommand{\P}{\mathbf{P}}

\newcommand{\F}{\mathcal{F}}

\newcommand{\al}{\alpha}
\newcommand{\be}{\beta}
\newcommand{\dl}{\delta}

\newcommand{\eps}{\varepsilon}

\newcommand{\g}{\gamma}

\newcommand{\s}{\sigma}

\newcommand{\ft}{\widehat}

\newcommand{\wt}{\widetilde}

\newcommand{\dx}{\partial_x}

\newcommand{\dt}{\partial_t}
\newcommand{\dtp}{\partial_{t'}}

\newcommand{\ta}{\theta}

\newcommand{\vda}{v^\dagger}
\newcommand{\wda}{w^\dagger}
\newcommand{\Fda}{F^\dagger}

\newcommand{\tv}{\wt v}
\newcommand{\tvda}{\tilde v^\dagger}

\newcommand{\tw}{\wt w}
\newcommand{\twda}{\wt w^\dagger}
\newcommand{\tE}{\wt E}

\renewcommand{\l}{\ell}

\renewcommand{\O}{\Omega}

\newcommand{\les}{\lesssim}
\newcommand{\ges}{\gtrsim}

\newcommand{\jb}[1]
{\langle #1 \rangle}

\newcommand{\ind}{\mathbf 1}

\newcommand{\M}{\mathcal{M}}

\newcommand{\N}{\mathbb{N}}

\newtheorem*{ackno}{Acknowledgements}

\renewcommand{\H}{\mathcal{H}}

\def\sgn{\textup{sgn}}

\newcommand{\Pbhi}{\mathbf{P}_{\textup{hi}} }
\newcommand{\Pblo}{\mathbf{P}_{\textup{lo}} }
\newcommand{\Pbhip}{\mathbf{P}_{\textup{+,hi}} }

\newcommand{\PbLO}{\mathbf{P}_{\textup{LO}}}

\numberwithin{equation}{section}
\numberwithin{theorem}{section}

\makeatletter
\@namedef{subjclassname@2020}{\textup{2020} Mathematics Subject Classification}
\makeatother

\begin{document}
\baselineskip = 14pt

\title[Unconditional deep-water limit for ILW]{Unconditional deep-water limit of the intermediate long wave equation in low-regularity}

\author[J.~Forlano, G.~Li, T.~Zhao]
{Justin Forlano, Guopeng Li and Tengfei Zhao }

\address{
Justin Forlano\\ 
School of Mathematics\\
Monash University\\
 VIC 3800\\
 Australia\\
 and\\
School of Mathematics\\
The University of Edinburgh\\
and The Maxwell Institute for the Mathematical Sciences\\
James Clerk Maxwell Building\\
The King's Buildings\\
Peter Guthrie Tait Road\\
Edinburgh\\ 
EH9 3FD\\
United Kingdom}

\email{justin.forlano@monash.edu}

\address{
Guopeng Li\\ 
School of Mathematics and Statistics\\
Beijing Institute of Technology\\
 Beijing\\
  100081\\
  China\\
  and 
School of Mathematics\\
 The University of Edinburgh\\ and The Maxwell Institute for the Mathematical Sciences\\
  James Clerk Maxwell Building\\
   The King’s Buildings\\
    Peter Guthrie Tait Road\\ 
    Edinburgh\\ EH9 3FD\\
     United Kingdom}

\email{guopeng.li@bit.edu.cn}

\address{
Tengfei Zhao\\ School of Mathematics and Physics\\
University of Science and Technology
Beijing\\
Beijing\\
100083\\
China}

\email{zhao\_tengfei@ustb.edu.cn}



%
%
%

\begin{abstract}

In this paper, we establish the unconditional deep-water limit of the intermediate long wave equation (ILW) to the Benjamin-Ono equation (BO) in low-regularity Sobolev spaces on both the real line and the circle. Our main tool is new unconditional uniqueness results for ILW in $H^s$ when $s_0<s\leq \frac 14$ on the line and $s_0<s< \frac 12$ on the circle, where $s_0 = 3-\sqrt{33/4}\approx 0.1277$. 
Here, we adapt the strategy of Mo\c{s}incat-Pilod (2023) for BO to the setting of ILW by viewing ILW as a perturbation of BO and making use of the smoothing property of the perturbation term.

\end{abstract}

\maketitle

\section{Introduction}
We consider the Cauchy problem for the
 intermediate long wave equation (ILW):
\begin{equation}
\left\{
\begin{aligned}
  & \dt u- \mathcal G_\dl \dx^{2}  u =\dx(u^2),\\
   & u|_{t=0}=u_0,
\end{aligned}
 \right. \label{ILW}
\end{equation}
where $0<\delta<\infty$,  $u:\R\times \M\to \R$
and $\M =\R$ or $
\T = \R/(2\pi \Z)$.
Here, the operator $\mathcal G_\dl$ is defined by   
$\mathcal G_\dl =\mathcal T_\dl -\delta^{-1} \dx^{-1}$, where $\mathcal T_\dl$ is given by
 $$ \ft{\Tdl f}(\xi)  = -i \coth( \delta \xi) \ft f(\xi), \quad \xi \in \widehat \M, $$
$\coth(x)$ is the hyperbolic cotangent function defined by 
$
\coth (x)=\frac{e^{x}+e^{-x}}{e^{x}-e^{-x}}, $ 
$x\in \R\backslash\{0\},$ with the convention $\coth(x)-\frac 1x=0$ for $x=0$,
and
$\widehat \M$ is the 
Pontryagin dual of $\M$, i.e., $\widehat \M=\R$, when $\M=\R$, and 
$\widehat \M=\Z$, when $\M=\T.$

ILW \eqref{ILW} was introduced in \cite{JE,KKD} as a model for the propagation of the internal wave
 of an interface
 in a stratified fluid of finite depth $\delta$. 
It is 
completely integrable and  
is, in a sense, a physically relevant intermediary model between the Korteweg–de Vries equation (KdV) and the Benjamin-Ono equation (BO).
Indeed, formally taking the shallow-water limit
($\delta \to 0$, and after a suitable rescaling), we have that \eqref{ILW} converges to KdV: 
\begin{align}
\dt v+ \dx^3 v=\dx(v^2). \label{KdV}
\end{align}
See Remark~\ref{RMK:shallow} for further details.
 In the deep-water limit 
($\delta \to \infty$), we see directly that the operator
$\Tdl$ tends to the Hilbert transform $\H$ (with Fourier multiplier $-i\,\sgn(\xi)$) and hence \eqref{ILW} reduces to the BO equation:
\begin{align}
\dt v-\mathcal H \dx^2 v=   \dx(v^2),
\label{BO}
\end{align}
see \cite{SAF, 
ABFS,
Saut2019,
Gli,Gli2, CLOP}. 

The validity of such limits has been numerically observed \cite{PD} and their rigorous justification is one of the central important questions related to ILW; see \cite{Saut2019, KS} for nice surveys. In this paper, we focus on the deep-water limit ($\dl \to \infty$) corresponding formally to the convergence of ILW to BO. Our main goal is to rigorously prove the unconditional convergence as $\dl\to \infty$ in low-regularity; see Theorem~\ref{THM:conv} and the discussion below.

Before we can hope to tackle this problem, we first need a satisfactory well-posedness theory for the limiting equation, namely BO, and the family of ILW equations for each $0<\dl<\infty$. Historically, this was a difficult problem as both ILW and BO are quasilinear \cite{MST} at every $L^2$-based Sobolev regularity, which nullifies any argument based on the contraction mapping theorem.
Nonetheless, by now the well-posedness situation for BO on both $\R$ and $\T$ has been settled \cite{Iorio, KTz1, KK, TAO04,
 IK, Moli1, BP, Moli2, MP, MV, IT1, 
GKT-1,GKT-2, KLV}. Namely, BO is well-posed in $H^{s}(\M)$ for any $s>-\frac 12$ and this is sharp \cite{BLin, GKT-2}. In comparison, the situation for ILW is less well-understood, we refer to \cite{ABFS, MV, MPV, Gli}. Nonetheless, there have been significant breakthroughs in the low-regularity well-posedness of ILW. Ifrim-Saut \cite{IS} proved well-posedness in $L^2(\R)$ (and small data long-time dispersive decay estimates) by adapting the argument for BO in \cite{IT1}. Namely, the authors in \cite{IS} use the BO gauge transform and the (quasilinear) normal form method. A key idea was to view ILW as a perturbation of BO in the following sense: 
after a Galilean transform which removes the term $\dl^{-1}\dx$ (see \eqref{Gal}) we may write \eqref{ILW} as 
\begin{align}
\dt u-\mathcal H \dx^2 u=   \dx(u^2)+\Qdl \dx u. \label{eq-v-intro}
\end{align}
The operator $\Qdl:=(\mathcal{T}_{\dl}-\H)\dx$ enjoys a strong smoothing property which suggests that this term is relatively harmless; see Lemma~\ref{Q-est}.
Combining this idea with the unified $L^2$-well-posedness argument for BO from \cite{MP}, the second author with A. Chapouto, T. Oh, and D. Pilod \cite{CLOP}, established well-posedness in $L^2(\M)$ in both geometries $\M=\T$ and $\R$. 

Despite the absence of a scaling symmetry for ILW, the first and second authors, with A. Chapouto, T. Oh, and D. Pilod \cite{CFLOP-2}, identified $H^{-\frac 12}(\M)$ as the critical space for ILW with any depth parameter. They showed ill-posedness when $s<-\frac{1}{2}$ in the sense of failure of continuity of the data-to-solution map, and proved a-priori bounds on smooth solutions for $-\frac 12<s<0$. The latter result relied upon the complete integrability of BO, and interestingly not that of ILW. It remains an important open problem to establish the well-posedness of ILW in negative Sobolev spaces.

Let us now return to the question of the deep-water limit of solutions of ILW to those of BO as $\dl\to \infty$. Our main result in this regard is the following:

\begin{theorem}\label{THM:conv}
Let $\M =\R$ or $\T$ and $s_0<s<\frac 12$, where $s_0 := 3-\sqrt{33/4}\approx 0.1277$. 
Given $u_0\in H^{s}(\M)$, let $\{u_{0,\dl}\}_{1\leq \dl <\infty}\subset H^{s}(\M)$ such that $u_{0,\dl}$ converges to $u_0$ in $H^{s}(\M)$ as $\dl\to \infty$. 
Let $u\in~C(\R;H^{s}(\M))$ denote the unique global-in-time solution to \textup{BO}~\eqref{BO} satisfying $u\vert_{t=0}=u_0$ and let $u_{\dl}\in C(\R;H^{s}(\M))$ be the unique global-in-time solution to \textup{ILW}~\eqref{ILW} with depth parameter $\dl$ and $u_{\dl}\vert_{t=0}=u_{0,\dl}$, where the uniqueness for both the \textup{BO}~\eqref{BO} and the \textup{ILW}~\eqref{ILW} solutions hold in the entire class $C(\R;H^{s}(\M))$.
Then, $u_{\dl}$ converges to $u$ in $C(\R;H^{s}(\M))$ as $\dl\to\infty$\footnote{Here, we endow 
$C(\R;H^{s}(\M))$ with the compact-open topology in time.
}. \end{theorem}

Theorem~\ref{THM:conv} shows that the 
 convergence in the deep-water limit $(\dl\to \infty)$ of ILW to BO occurs \textit{unconditionally}: if $\{u_{\dl}(0)\}_{1\leq \dl <\infty}$ converges to $u(0)$, then the \textit{entire} net of solutions to ILW $\{u_{\dl}\}_{1\leq \dl <\infty}$, converges to \textit{the} solution $u$ of BO in the whole class $C(\R;H^{s}(\M))$, irrespective of how $u_{\dl}$ or its limit $u$ are constructed and of any additional auxiliary information on them. We note that the previous deep-water convergence results in $H^{s}(\M)$ for $s\geq \frac 12$, shown in \cite{ABFS, Gli, Gli2, CLOP}, do occur unconditionally since, as we discuss below, the unconditional uniqueness in $C(\R;H^{s}(\M))$ of BO and ILW was known for any $s\geq \frac 12$.
 
 In \cite{CLOP}, the authors established that the deep-water limit occurs \textit{conditionally} down to $L^2(\M)$: the ILW solutions they construct converge to those solutions of BO in \cite{MP}.
We point out that when $s\leq \frac 14$ on $\M=\R$ or $s<\frac{1}{2}$ on $\M=\T$, the convergence in Theorem~\ref{THM:conv} is not implied by their result since this convergence is predicated on the fact that their solutions also belong to some auxiliary function spaces and are not just members of $C(\R;H^{s}(\M))$. We also refer to Remark~\ref{RMK:conv} for a brief discussion on the regularity restriction in Theorem~\ref{THM:conv}.

The key tool for upgrading the deep-water limit from a conditional to an unconditional statement is the unconditional uniqueness of solutions to ILW \eqref{ILW} (and BO) in $C(\R;H^{s}(\M))$. Namely, given any $u_0\in H^{s}(\M)$, there is a unique solution to \eqref{ILW} in $C(\R;H^{s}(\M))$ with initial data $u_0$, where uniqueness holds in the entire class $C(\R;H^{s}(\M))$. The notion of unconditional uniqueness was introduced by Kato in \cite{Kato} and implies that the solutions do not depend on how they are constructed. 
 Here, solutions to \eqref{ILW} are understood to belong to $C([0,T]; H^{s}(\M))$ and satisfy \eqref{ILW} in the sense of distributions on $(0,T)\times \M$ (and hence also \eqref{eq-v-intro}) for any $T>0$. 
In order for solutions to ILW to satisfy the equation \eqref{ILW} in the sense of distributions, we need $u(t)\in L^{2}_{\text{loc}}$, and thus we expect unconditional uniqueness for BO and ILW to hold down to $L^2$.
 
 As BO is better understood compared to ILW, there are more unconditional uniqueness results for BO than for ILW. The well-posedness result of Burq-Planchon~\cite{BP} gave unconditional uniqueness in $H^{\frac 12}(\R)$, followed by Molinet-Pilod \cite{MP} whose results lowered this to $s>\frac 14$ on $\R$ and $s\geq \frac{1}{2}$ on $\T$. On the circle, Kishimoto \cite{Kishimoto} pushed this down to $s>\frac 16$, which appeared to be almost sharp since the equation satisfied by the BO gauge transform has a cubic nonlinearity thus needing $u\in H^{\frac 16}\subset L^3$. Recently, Mo\c{s}incat-Pilod \cite{MP-2023} identified that this is not the case and established unconditional uniqueness in both geometries for $s> s_0$ where $s_0<\frac 16$.
We will comment on their method of proof later.

For ILW, the general well-posedness result of Molinet-Vento \cite{MV}
showed unconditional uniqueness for \eqref{ILW} for any $s>\frac12$ on both the real line and the circle. 
As \cite{CLOP} adapts the approach in \cite{MP} to \eqref{eq-v-intro}, they also obtained unconditional uniqueness in $H^{s}(\R)$ for $\frac 14<s\leq \frac 12$ and $H^{\frac 12}(\T)$.
In this paper, we further improve the unconditional uniqueness theory for \eqref{ILW} to the point that it agrees with that for BO on both $\R$ and $\T$.

\begin{theorem}\label{main-thm}
Let $\M=\R$ or $\T$, $0<\delta<\infty$, and $s>s_0$ where $s_0$ is as in Theorem~\ref{THM:conv}. Then, ILW \eqref{ILW} is unconditionally globally well-posed in $C(\R;H^{s}(\M))$. 
More precisely, for any initial data $u_0\in H^s (\M)$, there exists a unique global-in-time solution $u$ to the ILW \eqref{ILW} in the class $C(\R;H^s (\M))$ with $u\vert_{t=0}=u_0$.
\end{theorem}

Theorem~\ref{main-thm} extends the known unconditional uniqueness on the line to $s_0<s\leq \frac14$ and to $s_0<s<\frac{1}{2}$ on the circle. Interestingly, these results show that $s=\frac 16$ is also not a barrier for the unconditional uniqueness of solutions to ILW.

With Theorem~\ref{main-thm} in hand, we present a proof of Theorem~\ref{THM:conv}.
Let $u_{0}$ and $\{u_{0,\dl}\}_{1\leq \dl <\infty}$ be as in the statement of the theorem. Let $u\in C(\R;H^{s}(\M))$ be the solution to BO constructed in \cite{IK, Moli2,MP} with $u\vert_{t=0}=u_0$. By the unconditional uniqueness results in \cite{MP-2023}, $u$ is unique solution to BO with this initial data in $C(\R;H^{s}(\M))$. Similarly, let $u_{\dl}$ denote the solution to ILW such that $u_{\dl}\vert_{t=0} =u_{0,\dl}$ constructed in \cite{IS, CLOP}. This time, Theorem~\ref{main-thm} guarantees that this is the unique solution with this initial data in $C(\R;H^{s}(\M))$ for each fixed $0<\dl<\infty$. We then apply the conditional deep-water limit result in \cite[Theorem 1.2]{CLOP} to obtain the convergence of $u_{\dl}$ to $u$ in $C(\R;H^{s}(\M))$.

We now briefly discuss the strategy of the proof of Theorem~\ref{main-thm}. 
We first consider solutions to \eqref{eq-v-intro}, utilising the idea of viewing ILW as a perturbation of the BO equation. The general strategy is to then adapt the approach of Mo\c{s}incat-Pilod~\cite{MP-2023} on unconditional uniqueness for BO (essentially, \eqref{eq-v-intro} with $\dl=\infty$), which itself is inspired by the argument of Kishimoto~\cite{Kishimoto}. The goal is to compare two solutions $u$ and $u^{\dag}$ to ILW with the same initial data.

First, we use (a very slightly modified form of) the gauge transform for BO introduced by Tao \cite{TAO04}, and in the form by Burq-Planchon~\cite{BP}. Namely, we pass from $u,u^{\dag}$ to their gauged variants $w, w^{\dag}$, respectively, defined in \eqref{w} (see also \eqref{Gal}), which satisfy \eqref{eq-w} on $\R$. See Section~\ref{SEC:T} for the slightly different periodic setting. This equation is essentially the same as that for the gauged BO equation except for the addition of the new term $i\dx \Pbhip[ e^{iF} \Qdl v]$. It is not difficult to obtain control on the difference $u-u^{\dag}$ in terms of the difference $w-w^{\dag}$; see Lemma~\ref{v<w+v}. 

The main difficulty is obtaining the reverse: control on $w-w^{\dag}$ by $u-u^{\dag}$ with small constants; see Proposition~\ref{PROP:wtov}.
The proof of this result is responsible for the regularity restriction and is based on two normal form reductions (i.e. integration-by-parts in time), following those in \cite{MP-2023}. We point out that the normal form method has been a powerful approach to showing unconditional uniqueness; see for example \cite{BIT, GuoKO, K19}. The key idea in \cite{MP-2023} to go beyond the restriction $s>\frac 16$ is the addition of the refined Strichartz estimates which allows one to exploit the presence of a time integral for better integrability in space, see Lemma~\ref{re-Stri}, and Lemma~\ref{LEM:StrT}.

Whilst we follow the approach in \cite{MP-2023}, there are a few key differences. Firstly, we need to deal with the extra terms on the far right-hand side of \eqref{eq-v-intro} in view of the finite depth parameter $0<\dl<\infty$. This requires $\Qdl$ to be sufficiently smoothing; see Remark~\ref{RMK:twodepth}. Secondly, when $\M=\R$, there are new issues, not present in \cite{MP-2023}, in justifying the normal form reductions. 
In the usual approach, one applies the identity $e^{i\phi t}=( i \phi)^{-1}\frac{d}{dt} e^{i\phi t} $, where we assume that $\phi=\phi(\xi,\xi_1,\xi_2)$ does not vanish and $|\phi|$ is bounded from below, Fubini's theorem, and integration-by-parts to show that
\begin{align}
\begin{split}
\int_{0}^{t} \intt_{\xi=\xi_1+\xi_2} e^{i \phi t'} f(t',\xi_1)g(t',\xi_2)d\xi_1 dt'  =& \intt_{\xi=\xi_1+\xi_2} \frac{1}{i \phi} e^{i \phi t} f(t',\xi_1)g(t',\xi_2)d\xi_1 dt'  \bigg\vert_{t'=0}^{t'=t} \\
& -\int_{0}^{t} \intt_{\xi=\xi_1+\xi_2} \frac{1}{i\phi} e^{i \phi t'} \dt \big[f(t',\xi_1)g(t',\xi_2)\big]d\xi_1 dt'  
\end{split} \label{general}
\end{align}
for $f,g\in C^{1}(\R; L^2 (\R))$. Now, the $C^1$ assumption on both $f$ and $g$ allows us to use the product rule in time to distribute $\dt$ as 
\begin{align}
\dt \big[f(t,\xi_1)g(t,\xi_2)\big] = (\dt f)(t,\xi_1)g(t,\xi_2) + f(t,\xi_1)(\dt g)(t,\xi_2). \label{product}
\end{align}
In our setting, we need to apply the above general computation with $f=\ft w$, where $w$ is the gauged solution in \eqref{w} with spatial Fourier transform $\ft w$.
However, the extra term $i\dx \Pbhip [e^{iF} \Qdl v]$ in \eqref{eq-w} prevents us from showing $\ft w(t,\xi)$ is continuously differentiable in time for each fixed frequency $\xi$; we refer to Remark~\ref{RMK:wC1} for further details. Consequently, we can no longer justify the product rule as in \eqref{product} as in the usual approach above. 
Inspired by \cite[Appendix A]{K19}, we proceed more carefully 
splitting the gauge variable into a good part (which is $C^1$-in-time) and a bad part (which is not). 
For the contribution due to the bad part, we prove the analogue of the identity \eqref{general} directly rather than through the integration-by-parts in-time followed by the product rule approach as above.  As we do two rounds of the normal form reductions, we need this justification twice, first in Lemma~\ref{IBP-1} and secondly in Lemma~\ref{IBP-2}.

As in Mo\c{s}incat-Pilod \cite{MP-2023},
the arguments in this paper also imply a nonlinear
smoothing estimates for the gauged variable $w$.

\begin{corollary}\label{main-cor}
Let $s>s_0$ and $0<\delta<\infty.$ Let $u_0\in H^{s}(\M)$ and $u$ be the unique solution to ILW \eqref{ILW} with $u\vert_{t=0}=u_0$. Let $w$ denote the corresponding gauged variable defined by \eqref{w} and solving \eqref{eq-w} when $\M=\R$, or defined by \eqref{wperiodic} and solving \eqref{eq-w-T-m} when $\M=\T$.
In the periodic setting, we further assume $\int_{\T} u_0 dx = 0$.
Then, there exists $\eps>0$ and $C:\R_{+}^{3}\to\R_{+}$ such that for any $T>0$, it holds that
\begin{align*}
\| w(t)- e^{-it\dx^2}w_0\|_{C([0,T]; H^{s+\eps})} &\leq C(T, \dl, \|u_0\|_{H^s}) \quad \textup{if} \quad \M=\R, \\
\| w(t)- e^{-it(\dx^2+m_0)}w_0\|_{C([0,T];H^{s+\eps})} &\leq C(T, \dl, \|u_0\|_{H^s}) \quad \textup{if} \quad \M=\T,
\end{align*}
where $m_0 = \frac{1}{2\pi} \int_{\T} u_0^{2}(x)dx$.
\end{corollary}

The proof of Corollary~\ref{main-cor} follows from the smoothing estimates in Section~\ref{NF}, the resulting normal form equation for the gauged variable $w$ in \eqref{w}, and the uniqueness result of Theorem~\ref{main-thm} which allows us to use the global-in-time a priori bound on the $H^s$-norm of
solutions to ILW in \cite{IS, CLOP}. We refer the reader to the proof of
 \cite[Corollary 1.2]{MP-2023} for more details.
 
 \begin{remark}\rm \label{RMK:conv}
 As we can make sense of the nonlinearity in both \eqref{ILW} and \eqref{BO} distributionally when the solutions belong to $L^2(\M)$ in space, we expect that the unconditional deep-water limit should also occur down to $L^2(\M)$. The regularity restriction $s>s_0$ in Theorem~\ref{THM:conv} is due to the restriction in Theorem~\ref{main-thm}. Thus, any improvement on the restriction of Theorem~\ref{main-thm} yields an improvement in that of Theorem~\ref{THM:conv}. 
In particular, it seems possible that further normal form reductions may lower the threshold $s_0$ closer to $s=0$.
\end{remark}

 \begin{remark}\rm \label{RMK:twodepth}
 The approach of viewing ILW as a perturbation of BO is quite robust in the sense that we typically do not need to assume that the operator $\Qdl$ is infinitely smoothing. Indeed, in this paper, we only need smoothing of order $\frac 32+\eps$, for any $\eps>0$. Consequently, Theorem~\ref{main-thm} and Theorem~\ref{THM:conv} extend to 
 solutions of the ILW equation with two depth parameters, introduced in \cite{KKD}:
 \begin{align*}
 \dt u- c_1\mathcal G_{\dl_1} \dx^{2}  u-c_2\mathcal G_{\dl_2} \dx^{2}  u =\dx(u^2),
\end{align*}
 where $c_1,c_2>0$ and $0<\dl_1,\dl_2<\infty$.
  See also \cite[Remark 1.3]{CFLOP-2}.
 \end{remark}

  \begin{remark}\rm \label{RMK:shallow}
The convergence in the shallow-water setting ($\dl \to0$) requires a rescaling procedure, as introduced in \cite{ABFS}. 
Given a solution $u\in C(\R;L^2(\M))$ to ILW \eqref{ILW}, the rescaled function $\wt{u}(t,x)  =\tfrac{1}{\dl} u(\tfrac{1}{\dl}t, x)$
solves the scaled ILW equation (sILW)
\begin{align}
 \dt \wt{u}-\tfrac{1}{\dl} \mathcal G_\dl \dx^{2}  \wt{u} =\dx(\wt{u}^2). \label{sILW}
\end{align}
Note that $u$ solves ILW \eqref{ILW} with initial data $\dl u_0$ if and only if $\wt{u}$ solves sILW \eqref{sILW} with initial data $u_0$. In particular, Theorem~\ref{main-thm} implies that the sILW equation \eqref{sILW} is unconditionally globally well-posed in $C(\R;H^{s}(\M))$ for $s>s_0$.

The shallow water limit to KdV is then taken with respect to the Cauchy problem for sILW, which has been rigorously proved in $H^{s}(\M)$ for $s>\frac{1}{2}$ \cite{ABFS, Gli}. Moreover, as KdV \eqref{KdV} is unconditionally globally well-posed in $C(\R; H^{s}(\M))$ for any $s\geq 0$ and on both $\M=\R$ and $\T$ \cite{Zhou, BIT}, the shallow-water limit for $s>\frac 12$ occurs unconditionally. Given any improvement in the regularity restriction for the shallow-water limit, our unconditional uniqueness result for sILW would immediately yield an improvement in the regularity restriction for the unconditional shallow-water limit.
 \end{remark}

Finally, we give an overview of the rest of this paper. In Section~\ref{SEC:Prelim}, we discuss the BO gauge transformation on $\R$ and how to adapt it to ILW (after a Galilean transform). We derive the gauged ILW equation and control certain harmless parts of its nonlinearity. We then discuss the refined Strichartz estimates adapted for ILW. In Section~\ref{NF}, we begin normal form reductions on the gauged equation, with the goal of proving Proposition~\ref{PROP:wtov} which controls the difference of gauged solutions. Lastly, in Section~\ref{SEC:T} we discuss the modifications to deal with the periodic setting.

\section{Gauge transform and Strichartz estimates}\label{SEC:Prelim}

\subsection{Notation and preliminary estimates}

We use $A\les B$ to denote $A\leq C B$ for some constant $C>0$. 
We denote the Fourier transform by
$\F f(\xi) = \ft f(\xi) = \int_{\M} e^{-ix\xi} f(x) dx$, for $\xi\in\ft \M$,
and endowed with the Lebesgue measure if $\M=\R$ or counting measure if $\M=\T$.
For $s\in \R$, we denote the Sobolev space 
$H^s(\M)$ with the norm
\[
\|f\|_{H^s (\M)}
= \|J^s f\|_{L^2}
=\| \jb{\xi}^{s} \ft f(\xi)\|_{L^{2}_{\xi}(\ft \M)},
\]
where $\F J^s f= \jb{\xi}^s \ft f(\xi)$, with $\jb{\xi}=(1+\xi^2)^\frac12.$

Let $\psi:\R\to [0,1]$ be a smooth function  with support in $[-2,2]$ and equal to $1$ on $[-1,1]$.
Define $\psi_N(\xi)=\psi(\frac{\xi}{N})-\psi(\frac{2\xi}{N})$. 
For $N\in 2^{\Z}$, we denote 
 the Littlewood Paley projectors by
 \begin{align*}
  \F \,(\P_{\leq N} f)(\xi)  = \psi(\tfrac\xi{N})\ft f(\xi), \quad  \F\, \P_N f = \F \,\P_{\leq N} f -\F\, 
\P_{\leq \frac{N}{2}} f =\psi_N \ft f,
 \end{align*}
 and $\P_{>N}:= 1-\P_{\leq N}$. We then set
\begin{align}
\begin{split}
 \F\,( \P_{+} f)(\xi) = \ind_{\xi> 0} \ft f(\xi), &\quad   \F\, ( \P_{-} f)(\xi)  = \ind_{\xi< 0} \ft f(\xi),\\
  \Pblo = \P_{\leq 1}, \quad \PbLO =P_{\leq 2} \quad  \Pbhi &=1-\Pblo, \quad \Pbhip= \P_{+} \Pbhi.
  \end{split} 
  \label{projs}
\end{align}
 We also recall the frequency localised Sobolev embedding 
\begin{align*}
\| \P_N f\|_{L^{\infty} (\M)} \les N^{\frac 12}\|\P_{N}f\|_{L^2 (\M)}. 
\end{align*}



\begin{lemma}[Fractional Leibniz rule]
\label{prod-1}
Let $\M=\R$ or $\T$, $s>0$ and $1<p_j,q_j,r\leq \infty$, $j=1,2$,
such that $\frac1r=\frac{1}{p_j}+\frac{1}{q_j}$. 
Then we have
\begin{align*}
  \| J^s(fg)\|_{L^r (\M)} \les \|J^s f\|_{L^{p_1}(\M)} \|g\|_{L^{q_1}(\M)}  + \|f\|_{L^{p_2}(\M)} \| J^s  g\|_{L^{q_2}(\M)}.
\end{align*}
\end{lemma}
\noi
For a proof of Lemma~\ref{prod-1}, see \cite{CW,GO,BL} on $\R$ and 
\cite{IK2,GKO,BOZ} for $\T$.
Moreover, we recall the following lemma for product estimates from \cite[Lemma 2.7]{MP} and \cite[Lemma 3.3]{CLOP}.

\begin{lemma}\label{prod-2}
Let $\M=\R$ or $\T$, $2\leq q <\infty$, and $0\leq s\leq \frac12$.
Assume $F_1, F_2$ are two real-valued functions such that 
$\dx F_j  =f_j$ for $j=1,2$, and $J^{s}g\in L^{q}$. Then
\begin{align}
\|J^s (e^{iF_1} g)\|_{L^q }
&\les 
(1+\|f_1\|_{L^2}) \|J^s g\|_{L^q}, \label{JS1} \\
\|J^s \big((e^{iF_1}-e^{iF_2}) g\big)\|_{L^q}
&\les 
\big( \|f_1-f_2\|_{L^2} +\|e^{iF_1}-e^{iF_2}\|_{L^\infty }(1+\|f_1\|_{L^2})
\big)
\|J^s g\|_{L^q}. \label{JS2}
\end{align}
\end{lemma}

%
%

\subsection{Gauge transformation on $\R$} \label{SEC:GaugeR}

In this section, we slightly modify the gauge transformation for BO at low-regularity on the line from \cite{TAO04, BP} to deal with the ILW equation \eqref{ILW}. The gauge transformation on $\T$ is simpler to construct due to the absence of very low frequencies and we discuss it later in Section~\ref{SEC:T}.

Fix $0<\dl<\infty$. Let $s\geq 0$ and $u\in C([0,T]; H^s(\R))=C_{T}H^{s}(\R)$ be a distributional solution to \eqref{ILW}. 
In order to remove the transport term $\dl^{-1}\dx$, we consider the Galilean transformation 
\begin{align}
v(t,x)=u(t,x+\delta^{-1} t). \label{Gal}
\end{align} 
Then, 
we have $v(t,x)\in C([0,T]; H^s (\R))$ and $v$ is a distributional solution to 
\begin{align}
\begin{cases}
\dt v-\mathcal H \dx^2 v=   \dx(v^2)+\Qdl \dx v, \\
v\vert_{t=0}  =u_0,
\end{cases}
\label{eq-v}
\end{align}
where 
$\Qdl=(\mathcal T_\dl- \mathcal H)\dx $.
We note here that $\Qdl$ is a smoothing operator. See \cite{IS, CLOP, CFLOP-2}.
\begin{lemma}\label{Q-est} 
Let $\M= \R$ or $\T$ and $0<\dl<\infty$. 
For any $s\geq 0$, we have 
\begin{align*}
 \dl   \| \Qdl f \|_{H^{s}(\M)}+ \dl^{2}   \| \Qdl \dx f \|_{H^{s}(\M)} \les  (1+\delta^{-s}) \|f\|_{L^2 (\M)}.
\end{align*}
\end{lemma}

To construct the gauge transformation following \cite{BP}, we first construct a primitive of $v$. Namely, a function $F=F[v]$ such that $\dx F= v$. Let $\psi \in C_{0}^{\infty}(\R)$ and $\int_{\R} \psi(x) dx=1$.
We set 
\begin{align}
F(t,x):=\int_{\R}
 \psi(y)
  \int_{y}^x v(t,z) dz dy+G(t), \label{def-F}
\end{align}
where 
$$G(t):=  
 \int_{0}^{t} \int _{\R} -\psi'(y)    \H  v(t',y) +  \psi(y)  \Qdl  v(t',y) +    \psi(y) v(t',y)^2  dy dt' . $$
It is clear that $\dx F=v$ and, moreover, from \eqref{eq-v}, $F$ satisfies
%
%
 %
\begin{align} 
  \dt F- \mathcal H \dx^2 F=v^2+\Qdl v. \label{eq-F}
\end{align}
Denote the Benjamin-Ono gauge transform of $F$ by
\begin{align}
  W=\Pbhip (e^{iF}) \quad \text{and} \quad  w=\dx W= \dx \Pbhip(e^{iF})=-i \Pbhip(e^{iF} v). \label{w}
\end{align}
Then, as in \cite[Section 3]{CLOP}, we arrive at the gauged ILW equation:
%
\begin{align}
 \begin{split}
   \dt w+i \dx^2 w & 
    = -2  \dx \Pbhip [(\dx^{-1} w)(  \P_-\dx v)]
   -2\dx \Pbhip [(\Pblo e^{iF}) ( \P_{-} \dx v)]\\
   & \quad + i\dx \Pbhip [e^{iF} \Qdl v]
 \end{split} \label{eq-w}
\end{align}
with $w\vert_{t=0}= \dx \Pbhip[ e^{iF[u_0]}]$.
The first two terms on the right-hand side of \eqref{eq-w} are the usual terms appearing in the gauged BO equation, see \cite[(2.11)]{MP}. The third term is the additional term arising from viewing the ILW equation as a perturbation of the BO equation. The smoothing property of $\Qdl$ in Lemma~\ref{Q-est} ensures that this term is mostly harmless, like the second term on the right-hand side of \eqref{eq-w}.
Indeed, we have the following.

\begin{lemma}\label{est-error}
Let $\M=\R$ or $\T$, $0<\dl<\infty$, $\sigma_1 \geq 0 $, and $0\leq \sigma_2< \frac12$.
Define 
\begin{align}
E_1(f,g)= -2\dx \Pbhip [(\Pblo f) ( \P_{-} \dx g)] \quad \text{ and } \quad 
E_2(f,g)=  i\dx \Pbhip [f \Qdl g]. \label{E1E2}
\end{align}
Assume that $F_1,F_2$ are two real-valued functions such that $\dx F_j =f_j$, $j=1,2$.
 Then,
  \begin{align}
  \begin{split}
\| E_1(e^{i F_1},g_1) -E_1(e^{iF_2},g_2) \|_{H^{\sigma_1}} \\
 &\hspace{-5cm}\les [(1+\|f_1\|_{L^2})\|F_1-F_2\|_{L^{\infty}}+\|f_1-f_2\|_{L^2}] \|g_1\|_{L^2}+(1+\|f_2\|_{L^2})\|g_1-g_2\|_{L^2},
\label{error-1} 
\end{split}\\
\begin{split}
     \| E_2(e^{iF_1},g_1) -E_2(e^{iF_2},g_2) \|_{H^{\sigma_2}} \\
& \hspace{-5cm} \les \dl^{-1}(1+\dl^{-1-\s_2}) 
 (1+\|f_1\|_{H^{\s_2}}+\|f_2\|_{H^{\s_2}})^2 \\
 & \hspace{-2cm} \times \big[ \|g_1\|_{L^2}( \|f_1-f_2\|_{H^{\s_2}}+\|F_1-F_2\|_{L^{\infty}})+ \|g_1-g_2\|_{L^{2}} \big].
  \end{split}
\label{error-2}
\end{align}
\end{lemma}

Before we present the proof of Lemma~\ref{est-error}, we recall the following from \cite[Remark 2.3]{MP-2023}. For a real-valued continuous function $F$, we always have $e^{iF}\in L^{\infty}(\R)\setminus L^{2}(\R)$. Thus, $e^{iF}$ is merely a tempered distribution and its Fourier transform is understood in the same sense. If $\dx F ~\in~L^2(\R)$, then $\dx(e^{iF})\in L^2(\R)$ and so $\Pbhi(e^{iF})$ is well-defined through
\begin{align}
\Pbhi(e^{iF}) = \dx^{-1} \Pbhi( \dx e^{iF}).  \label{dxexpF}
\end{align}
Note that \eqref{dxexpF} and Bernstein's inequality then give
\begin{align}
\| \Pbhi(e^{iF}) \|_{L^{\infty}_{x}} \les \|\dx F\|_{L^{2}_{x}}. \label{PhieF}
\end{align}
In general, $\Pblo(e^{iF})$ may not be well-defined, so instead, with a slight abuse of notation, we define $\Pblo(e^{iF}):= e^{iF} -\Pbhi(e^{iF})$. It then follows from \eqref{dxexpF} that 
\begin{align} 
\Pblo(e^{iF}) = 1+ \Pblo(e^{iF}-1) + \Pbhi(e^{iF}-1) -\Pbhi(e^{iF}) = 1+ \Pblo(e^{iF}-1), \label{PloeiF}
\end{align}
where $e^{iF}-1 \in L^2(\R)$ and so $\Pblo(e^{iF}-1)$ is well-defined with the frequency projection $\Pblo$ in \eqref{projs}. In particular, since $\P_{+}\P_{-}=0$, 
\begin{align}
E_{1}(e^{iF},g) = E_{1}(e^{iF}-1,g). \label{E11}
\end{align}

\begin{proof}[Proof of Lemma~\ref{est-error}]
We first show \eqref{error-1}. For $f\in L^2$, by frequency considerations, we have 
\begin{align}
E_{1}(f,g) = -2\dx \PbLO \Pbhip[ (\Pblo f)(\PbLO\P_{-}\dx g)]. \label{freqE1}
\end{align}
Then, by Bernstein's inequality and \eqref{PhieF}, we have 
\begin{align*}
\| E_1(e^{i F_1},g_1) -E_1(e^{iF_2},g_2) \|_{H^{\sigma_1}} & \les \|\Pblo(e^{iF_1}-e^{iF_2})\|_{L^{\infty}} \| g_1\|_{L^2} + \| \Pblo(e^{iF_2})\|_{L^{\infty}} \| g_1-g_2\|_{L^2}.
\end{align*}
Now \eqref{error-1} follows from
\begin{align*}
\|\Pblo(e^{iF_1}-e^{iF_2})\|_{L^{\infty}} \les (1+\|\dx F_1\|_{L^2})\|F_1- F_2\|_{L^{\infty}}+\|\dx F_1-\dx F_2\|_{L^2}.
\end{align*}
We move onto proving \eqref{error-2}.
We will only prove the following inequality
\begin{align}
\| E_{2}(e^{iF},g)\|_{H^{\s_2}} \les \dl^{-1}(1+\dl^{-1-\s_2})(1+ \|f\|_{H^{\s_2}})^2 \|g\|_{L^{2}}. \label{error-3}
\end{align}
Then \eqref{error-2} follows from \eqref{error-3} and using similar ideas and \eqref{JS2}. By \eqref{JS1}, we have
\begin{align*}
\| E_{2}(e^{iF},g)\|_{H^{\s_2}} & \les \|e^{iF} f \Qdl g\|_{H^{\s_2}}+\|e^{iF} \Qdl \dx g\|_{H^{\s_2}} \\
& \les (1+\|f\|_{L^2})\big(\| f \Qdl g\|_{H^{\s_2}}+\|\Qdl\dx g\|_{H^{\s_2}}\big).
\end{align*}
For the second term, we use Lemma~\ref{Q-est}, while for the first term, by Lemma~\ref{prod-1}, Sobolev embedding, we have 
\begin{align*}
\|f \Qdl g\|_{H^{\s_2}} &\les \|f\|_{H^{\s_2}}\|\Qdl g\|_{L^{\infty}}+ \|f\|_{L^2} \|\Qdl J^{\s_2}g\|_{L^{\infty}}\les \|f\|_{H^{\s_2}}\|\Qdl J^{\s_2+1}g\|_{L^2}.
\end{align*}
Another use of Lemma~\ref{Q-est} completes the proof of \eqref{error-3}.
\end{proof}



\begin{lemma}
\label{dif-F}
Let $0<\dl<\infty$ and $v,v^\dag\in C_TL^2 (\R)$ be two solutions
to \eqref{eq-v} with initial data $v(0)$ and $v^\dag(0)$, respectively. Assume that $\Pblo v(0) =\Pblo  v^\dag(0).$
 Let $F, F^\dag$ be the spatial primitives of $v, v^\dag$ constructed in \eqref{def-F} and satisfying \eqref{eq-F}, respectively. 
 Then, we have
\begin{align}
\|F(0)-F^\dag(0)\|_{L^\infty} &\les \| v(0)-v^\dag(0)\|_{L^2},\label{dif-F-1}\\
\|F-F^\dag\|_{C_T L^\infty} &\les \jb{T} (\|v\|_{C_TL^2} + \|v^\dag \|_{C_TL^2} +\dl^{-1} ) \| v-v^\dag\|_{C_TL^2}.
\label{dif-F-2}
\end{align}
\end{lemma}
\begin{proof}
The proof is essentially the same as that in \cite[Lemma 4.1]{MP}, where the low-frequency assumption ensures \eqref{dif-F-1} by Bernstein's inequality. For \eqref{dif-F-2}, we use $\Pbhi ( F-F^{\dag})= \dx^{-1}\Pbhi (v-v^{\dag})$ and Bernstein's inequality. For the low-frequencies, we use the Duhamel formulation \eqref{Duh-F}. The only new term is that with $\Qdl$ for which we use Lemma~\ref{Q-est}, picking up a factor of $\dl^{-1}$.
\end{proof}

For $v\in C_{T}L^{2}$ which is a solution to 
\begin{align}
  v(t)= e^{t \H \dx^2 } v(0) + 
  \int_0^t  e^{(t-t')\H \dx^2}
 \big (\dx(v^2)+\Qdl \dx v \big) dt' \label{Duh-v}
 \end{align}
 in the sense of spatial distributions for $t\in [0,T]$, we see that $F=F[v]$ in \eqref{def-F} solves 
 \begin{align}
F(t)= e^{t \H \dx^2 } F(0) + 
 \int_0^t  e^{(t-t')\H \dx^2}
 \big (v^2+\Qdl v \big) dt',\label{Duh-F}
 \end{align}
 in the sense of spatial distributions, for all $t\in [0,T]$, and where $F(0)= \int_{\R} \psi(y) \int_{y}^{x} v(0,z)dzdy$.
However, the formal computations deriving \eqref{eq-w} are only easily justified for sufficiently high-regularity $v\in C_{T}H^{s}$. For $v\in C_{T}L^2$, we need to justify that the gauged variable $w$ is a distributional solution to \eqref{eq-w}. This can be done by following the arguments in \cite[Remark 2.7]{MP-2023}, themselves originating in \cite{Kishimoto}. 
The idea is to consider frequency truncated $v_{\leq N}:= \P_{\leq N} v$ for some $N\in 2^{\N}$, define a corresponding spatial antiderivative $F_{\leq N}$ and gauged variable $w_{\leq N}$, and to pass to the limit in $C_{T}H^{-2}_{x}$ as $N\to \infty$ in the resulting equation for $w_{\leq N}$. 
In the presence of the depth parameter $0<\dl<\infty$, \eqref{error-2} with $\s_2=0$ allows to pass to the  limit in $C_{T}L^2_{x}$ for the last term on the right-hand side of \eqref{eq-w}. 
Moreover, by Lemma \ref{est-error}, the gauge transformation $w$ of the solution $v$ solves 
\begin{align*}
 w(t) & =e^{-it\dx^2 }w(0) + \int^t_0 e^{-i(t-t')\dx^2 }
  \big\{\hspace{-1mm} -\hspace{-1mm}2\dx\Pbhip [(\dx^{-1} w)( \P_-\dx v)]
  +E_{1}( e^{iF},v) + E_{2}(e^{iF},v)\big\} dt'
\end{align*}
  in the sense of spatial distributions for each $t\in [0,T]$.

\subsection{Strichartz estimates}

We recall the Strichartz estimates for BO on $\R$: if $u$ solves
$\dt u-\H \dx^2 u=f,$ on $[0,T]\times \R$, then, for any pair $(p,q)$ satisfying $4\leq p\leq \infty$,  $2\leq q\leq \infty$, and $\frac{2}{p}+\frac{1}{q}=\frac{1}{2}$,  we have
\[
\|u\|_{L^p_T L_{x}^q } \les \|u\|_{L^\infty_T L_{x}^2} +\|f\|_{L^1_TL_{x}^2}.
\]
We call such pairs of exponents $(p,q)$ Strichartz admissible.
We now state the refined Strichartz estimates we will use here, which are essentially proven in \cite[Lemma 2.8]{MP-2023} with the addition of using Lemma~\ref{Q-est}.

\begin{lemma}[Refined Strichartz estimates] \label{re-Stri}
Let $0<\dl<\infty$, 
$0\leq s \leq \frac14$, $N\in 2^{\N}$, $N\geq 2^6$, and $T>0$.
Let $(p,q)$ be a Strichartz admissible pair 
 and define  $\al(s,p)=\frac1p(\frac32-s)-s$.

\noi \textup{(i)} If $v$ is a solution to the equation
 $\dt v-\H \dx^2 v = \dx(v_1v_2+v_3v_4) + \Qdl \dx v_5 $,
  then

\begin{align*}
\begin{split}
\| \P_N v \|_{L^p_T L^q} & \les T^{\frac1p} N^{\al(s,p)}\big\{  \|\P_N v\|_{L^\infty_T H^s }  + \|v_1\|_{L^\infty_T H^s }  \|v_2\|_{L^\infty_T H^s }   \\
& \hphantom{XXXXXXXXXXX}
 + \|v_3\|_{L^\infty_T H_x^s }  \|v_4\|_{L^\infty_T H^s }+  \dl^{-2}\|\P_N v_5\|_{L^\infty_T L^2} \big\}.
 \end{split}
\end{align*}

\noi \textup{(ii)}
If $w$ is a solution to the equation
$$\dt w+i \dx^2 w = 
-2 \Pbhip \dx [ \dx^{-1}  w_1 \cdot \P_- \dx w_2
+\dx^{-1}  w_3\cdot  \P_- \dx w_4  ] + \phi, $$   
where $\supp \ft w_1 , \supp \ft w_3 \subset (2^{-3},\infty)$ and $\phi\in L^{\infty}_{T}L^{2}_x$, then we have
\begin{align}
\begin{split}
  \| \P_N w \|_{L^p_T L^q} 
  & \les T^{\frac1p} N^{\al(s,p)}
\big\{\|\P_N w\|_{L^\infty_T H^s }  + \|w_1\|_{L^\infty_T H^s }  \|w_2\|_{L^\infty_T H^s }  \\ 
&\hphantom{XXXXXXXXX}+ \|w_3\|_{L^\infty_T H^s }  \|w_4\|_{L^\infty_T H^s }+ \|\P_N \phi\|_{L^{\infty}_T L^2 } \big\}.
\end{split}
 \label{RS-2}
\end{align}
\end{lemma}

\noi
The estimates in Lemma~\ref{re-Stri} give rise to the following bounds.

\begin{lemma}\label{dt-est}
 Let $0<\dl<\infty$, $0<s\leq \frac14$, $N\in 2^{\N}$ such that $N\geq 2^6$, and
 $2\leq q \leq 4$ such that $(\frac32-s)(\frac14-\frac{1}{2q})-s <0$.
 
\noi \textup{(i)} If $v, \vda$ are two solutions to the Cauchy problem of $\eqref{eq-v}$ with the same initial data $v_0\in H^s$, and we denote
$\tv = e^{-t \H \dx^2 }v(t)$ 
and
 $\tvda = e^{-t \H \dx^2 }\vda(t)$, then 
\begin{align}
 \| \P_N \dt (\tv -\tvda) \|_{L^1_TL^2} 
 \les 
 T 
 N^{\frac{2}{q}+\frac12} 
 \big( 
 1+\dl^{-2}+\|v\|_{L^\infty_T H^s} 
 + \|\vda\|_{L^\infty_T H^s}  
  \big)^3
 \|v-\vda\|_{L^\infty_T H^s}.
 \label{dif-v}
\end{align}

\noi \textup{(ii)}
If $w, \wda$
are the gauge transformations of
$v,\vda$ in \textup{(i)}, respectively, and $\tw = e^{ it\dx^2 }w(t)$ 
and
 $\twda = e^{it\dx^2 } \wda(t)$, then
\begin{align}
\begin{split}
 \| \P_N \dt (\tw -\twda) \|_{L^1_TL^2} \les T N^{\frac{2}{q}+\frac12} 
 \big( 
 1+ \dl^{-1}+T+\|v\|_{L^\infty_T H^s} 
 + \|\vda\|_{L^\infty_T H^s}  
  \big)^6  \\
 \times \big(
 \|w-\wda\|_{L^\infty_T H^s} + 
 \|v-\vda\|_{L^\infty_T H^s}
 \big). 
 \end{split}
 \label{dif-w-2} 
\end{align}
\end{lemma}

\begin{proof}
The proofs of (i) and (ii) are almost exactly the same as those in \cite[Lemma 2.11]{MP-2023}.
For (i), the only new estimate required is
\begin{align*}
 \| 
  e^{-t \H \dx^2}  
  \P_N \Qdl \dx
 (v-\vda)
 \|_{L^1_T L^2} 
 \les 
 T \dl^{-2} \| v-\vda \|_{L^\infty_T L^2},
\end{align*}
and for (ii), we only need to estimate
\begin{align*}
\| e^{ -it \dx^2 }\P_N   \dx \Pbhip( e^{iF} \Qdl v -e^{i\Fda} \Qdl \vda )
  \|_{L^1_T L^2}. 
  \end{align*}
This is easily done by distributing the derivative, and using \eqref{JS2}, Lemma~\ref{Q-est}, and \eqref{dif-F-2}.
\end{proof}

The final result of this section is control of the solutions $v, v^{\dag}$ in terms of their gauge transformations $w,w^{\dag}$.

\begin{lemma}\label{v<w+v}
Let $0<\dl<\infty$, $0 \leq  s <\frac12$, $N\geq 10$ be a dyadic number, and $T>0$.
If $v, \vda$ are two solutions to $\eqref{eq-v}$ and $w,\wda$ are the corresponding
gauge transformations of $v,\vda$ respectively, then
\begin{align}
&\|\P_{\leq N} (v-\vda) \|_{C_T H^s}
\les T \big\{ N^{\frac{3}{2}+s}  (1+\|v\|_{C_TH^s} + \|\vda\|_{C_T H^s}  )^2 +\dl^{-2}(1+\dl^{-s})\big\} \|v-\vda\|_{C_T L^2 } ,\label{dif-v-lo} \\
\begin{split}
& \|\P_{> N} (v-\vda) \|_{C_T H^s}  \les   (1+\dl^{-1}+\|v\|_{C_TH^s} + \|\vda\|_{C_TH^s}  )^5\\ 
 & \hphantom{XXxXXXXXXX}\times \big(
  \|w-\wda\|_{C_TH^s}
   + 
    \jb{T} (N^{s-\frac12} + \|P_{>N} \wda \|_{C_TH^s})
\|v-\vda\|_{C_T L^2 } \big).
\end{split} \label{dif-v-hi}  
\end{align}
\end{lemma}
\begin{proof}
We follow \cite[Lemma 2.12]{MP-2023}. For the low-frequencies in \eqref{dif-v-lo}, we use the Duhamel formulation \eqref{Duh-v} for which the only new term can be estimated as follows using Lemma~\ref{Q-est}:
\begin{align*}
\| \mathbf{P}_{\leq N} \mathcal{Q}_{\dl}\dx (v-v^{\dagger} ) \|_{L^1_T H^s}\les \dl^{-2}(1+\dl^{-s})T\| v- v^{\dagger} \|_{C_T L^2}.
\end{align*}
For the high-frequencies in \eqref{dif-v-hi}, we use \eqref{dif-F-2} whenever such a difference appears in the proof of \cite[Lemma 2.12]{MP-2023}.
\end{proof}

\section{Normal form reductions on $\R$} \label{NF}

The main goal of this section is to prove the following result which is an estimate controlling the difference of the gauged variables $w$ in terms of the original variable $v$. This will provide a counterpart to Lemma~\ref{v<w+v}.

\begin{proposition}
\label{PROP:wtov}
Let $0<\dl<\infty$, $s_0<s\leq \frac14,$ where $s_0$ is as in Theorem~\ref{THM:conv},
 $0<T<1$, and $M>1$. 
If $v, \vda\in C_T H^s$ are two solutions 
to $\eqref{eq-v}$ with the same initial data
and $w,\wda$ are the corresponding
gauge transformations of $v,\vda$ respectively, then,
\begin{align} 
\begin{split}
  \| w-\wda\|_{C_T H^s} & \les
(T M^\frac{3}{2} +M^{-\frac1{16} } )
(1+ \dl^{-1} +\|v\|_{C_T H^s}+\| \vda \|_{C_TH^s} )^{10}  \\
 & \quad \quad \times (\| w-\wda \|_{C_T H^s} +\|v-\vda\|_{C_T H^s} ).
\end{split}
\label{dif-w}
\end{align}

\end{proposition}

The proof follows the similar case for the BO equation dealt with in \cite{MP-2023}. However, we must now control the additional terms arising from the depth parameter $0<\dl<\infty$. We postpone the proof of Proposition~\ref{PROP:wtov} as we must first make the necessary preparations for the normal form reductions we will employ. 
Given Proposition~\ref{PROP:wtov} and Lemma~\ref{v<w+v}, we can give the proof of Theorem~\ref{main-thm}, which follows that in \cite{MP-2023}.

 \begin{proof}[Proof of Theorem 
\ref{main-thm} when $\M=\R$]
Fix $0<\dl<\infty$ and $s_0<s\leq \frac 14$. 
It suffices to prove the unconditional well-posedness of solutions to the renormalised ILW equation \eqref{eq-v}. 
Given $u_0 \in H^{s}$, let $v\in C(\R,H^s)$ be the global solution to \eqref{eq-v} constructed in
  \cite{IS, CLOP} such that $v\vert_{t=0}=u_0$. Suppose that there is an open interval $I$ containing zero and another solution $v^{\dag}\in C(I;H^{s})$ to \eqref{eq-v} such that $v^{\dag}\vert_{t=0}= u_0$.  
By time translation symmetry and
time reversal symmetry of \eqref{eq-v}, it suffices to show 
$v \equiv \vda\in C_T H^s$ for some small $T>0$.
  Let $w$ and $\wda$ be the gauge transforms of $v$ and $\vda$ given by \eqref{w}, respectively.
Fix $T_1\in I\cap (0,1)$ and let
\begin{align*}
  K= (1+C_2) (1+\|v\|_{C_{T_1} H^s} + \|\vda\|_{C_{T_1} H^s}+\dl^{-1}),
\end{align*}
where 
$C_2 $ be the implicit constant in the estimate \eqref{dif-v-hi}.
Choose $N\in 2^{\mathbb{N}}$ so that
\begin{align}
 C_2 K^{5} (N^{s-\frac12}+ \|\P_{>\frac{N}2} \wda\|_{C_{T_1} H^s})  \label{choice1}
\leq \tfrac14. 
\end{align}
Next, we choose $0<T<T_1$ such that
\begin{align}
C_1 T \{  N^{\frac{3}{2}+s}K^2+ \dl^{-2}\} \leq \tfrac14, \label{choice2}
\end{align}
where $C_1$ is the implicit constant in the estimate \eqref{dif-v-lo}. 
Then, by \eqref{choice1}, \eqref{choice2}, and Lemma \ref{v<w+v}, we have 
\begin{align}
\|v-\vda\|_{C_T H^s} \leq 2C_{2}K^5 \|w-\tw\|_{C_TH^s}. \label{dv<dw}
\end{align}

Now we use Proposition~\ref{PROP:wtov} to obtain the reverse difference estimate. 
Fix $0<\eta<\frac12$ to be chosen later.
We choose $M$, independently of $T$, so that $C_{3}M^{-\frac{1}{16}}K^{10} \leq \eta$, where $C_3$ is the implicit constant in \eqref{dif-w}. By reducing $T$, if necessary, so that $C_{3}TM^{\frac 32}K^{10}\leq \eta$, \eqref{dif-w} implies 
\begin{align}
\| w-\wda\|_{C_{T}H^{s}} \leq \tfrac{2\eta}{1-2\eta}\|v-\vda\|_{C_{T}H^{s}}. \label{dw<dv}
\end{align}
Finally, we choose $\eta$ sufficiently small so that 
\begin{align}
2C_{2}K^{5}\tfrac{2\eta}{1-2\eta} \leq \tfrac 12. \label{choice3}
\end{align}
Combining \eqref{dv<dw}, \eqref{dw<dv}, and \eqref{choice3} then shows that $v\equiv \vda$ on $[0,T]$, which completes the proof of Theorem \ref{main-thm} when $\M=\R.$
 \end{proof}

\subsection{Preparations for normal forms}

Let 
$v\in C_{T}H^{s}$  be a solution to \eqref{eq-v} and 
$w$ be the gauge transformation of $v$ defined in \eqref{w}. Then, $w$ solves \eqref{eq-w}. We pass to the interaction representation by defining $\tw=e^{it \dx^2 }w$, $\tv=e^{-t\H \dx^2}v(t)$,
 and $\tE=e^{it \dx^2} E(t)$, where 
 \begin{align}
E:=E_1(e^{iF},v)+E_2(e^{iF},v) \label{Edef}
\end{align}
  and $E_1,E_2$ were defined in \eqref{E1E2}.
Then we have 
\begin{align} 
 \ft{\wt w}(t')\big\vert_{t'=0}^{t'=t} = \hspace{-1mm}-2i \int_0^t \intt_{\xi=\xi_{12}} 
e^{-it' \Omega(\xi,\xi_1,\xi_2) } 
 \tfrac{\xi \xi_2}{\xi_1}
 \sigma(\xi,\xi_1,\xi_2)   
\ft{\tw} (t',\xi_1)\ft{\tv} (t',\xi_2) d\xi_1  dt' + \int_{0}^{t}  \ft{\tE}(t',\xi)dt',
  \label{eq-tw}
\end{align}
where $\xi_{12}:=\xi_1+\xi_2$, 
\begin{align}
\Omega(\xi,\xi_1,\xi_2)= \xi|\xi|-\xi_1|\xi_1|-\xi_2|\xi_2| \quad \text{and} \quad  \sigma(\xi,\xi_1,\xi_2)=\chi_+(\xi) \wt \chi_+(\xi_1) \ind_{<0}(\xi_2).  \label{Os}
\end{align}
Here,
$\chi_+$ is the symbol of $\Pbhip$ 
and $\wt \chi_+ $  is a Schwartz function such that $\wt \chi_+=1$ on the support of $\chi_+$
 and  $\wt \chi_+=0$ on a neighbourhood of $0$.  
Note that by the support restriction of $\sigma$,
we have $\Omega(\xi,\xi_1,\xi_2)=2\xi\xi_2$.
On the other hand, we have 
\begin{align}
\begin{split}
\ft \tv(t',\xi)\big\vert_{t'=0}^{t'=t} & =i\xi \int_0^t \int_{\xi=\xi_{12}} e^{-it'\Omega(\xi,\xi_1,\xi_2)}
\ft \tv(t',\xi_1) \ft \tv(t',\xi_2) d\xi_1 dt'  +  \int_0^t \F (\Qdl \dx v)(t',\xi) dt'.
\end{split}\label{eq-tv}
\end{align}

\noi
We use the following lemma to justify the normal form steps for $v$.
\begin{lemma}\label{conti}
 Let  $v\in C_T L^2$ solve \eqref{eq-v}.
Then, for each fixed  $\xi\in \R$, the functions
 $t~\to~\ft \tv(t,\xi),\, \ft v(t,\xi)$ are continuously differentiable. 
\end{lemma}
\begin{proof}
It suffices to show that $\ft \tv(t,\xi)$ is continuously differentiable.
 From \eqref{eq-tv}, we have 
 \begin{align}
\dt \ft \tv(t,\xi)  = i\xi \int_{\xi=\xi_{12}} e^{-it\Omega(\xi,\xi_1,\xi_2)}
\ft \tv(t,\xi_1) \ft \tv(t,\xi_2) d\xi_1  + i \xi \Qdl(\xi) \ft \tv(t,\xi). \label{eq-tv2}
\end{align}
The assumption $v\in C_{T}L^{2}_x$ with the dominated convergence theorem shows that 
\begin{align*}
t \mapsto  i\xi \int_{\xi=\xi_{12}} e^{-it\Omega(\xi,\xi_1,\xi_2)}
\ft \tv(t,\xi_1) \ft \tv(t,\xi_2) d\xi_1 
\end{align*}
is continuous for each fixed $\xi \in \R$. See \cite[Lemma 3.2]{MP-2023}. It follows from \eqref{eq-tv2} that
  \begin{align*}
\dt \big( e^{-i \xi \Qdl(\xi) t} \ft \tv(t,\xi) \big) &=   e^{-i \xi \Qdl(\xi) t}  \big[ \dt \ft\tv(t,\xi)-i \xi \Qdl(\xi) \ft \tv(t,\xi) \big]  \\
&=  e^{-i \xi \Qdl(\xi) t} i \xi \int_{\xi=\xi_{12}} e^{-it\Omega}\,
\ft \tv(t,\xi_1) \ft \tv(t,\xi_2) d\xi_1 . 
\end{align*}
  Here, the first equality is justified in the sense of temporal distributions for each fixed $\xi\in \R$, and the second equality shows that the map $t\mapsto  e^{i \xi \Qdl(\xi) t} \ft \tv(t,\xi) $ is $C^1$ for each fixed $\xi \in \R$. We thus conclude that $t\mapsto \ft \tv(t,\xi)$ is $C^{1}$ as well.
\end{proof}

 \begin{remark}\rm \label{RMK:wC1}
 Given $v\in C_{T}L^{2}$ a solution to \eqref{eq-v}, let $w$ be its gauge transformation according to \eqref{w}. Contrary to the BO equation in \cite{MP-2023}, we are not able to show that $t\mapsto \ft w(t,\xi)$ and $t\mapsto \ft \tw(t,\xi)$ are $C^{1}$ functions for each fixed $\xi\in \R$. 
The culprit is the term
$\F E_2(e^{iF},v)(t,\xi)$ for 
each fixed $\xi\in \R$, as we only know
 $E_2(e^{iF},v)\in C_T L^2$ and not that $E_2(e^{iF},v)\in C_T L^1$.
As we cannot apply the product rule directly in the normal form reductions for terms including $\ft \tw(t,\xi)$, we need extra ingredients in the normal form reductions below.
 \end{remark}

\subsection{First step of normal form reductions}\label{SEC:first}

Denote by $\mN^{(1)} $ the bilinear operator
\begin{align*}
\F \{\mN^{(1)}(\tw,\tv)\}(t,\xi)=
-2i \int_{\xi=\xi_{12}} 
e^{-it \Omega(\xi,\xi_1,\xi_2) }  \frac{\xi \xi_2}{\xi_1}
 \sigma(\xi,\xi_1,\xi_2)   
\ft{\tw} (t,\xi_1)\ft{\tv} (t,\xi_2) d\xi_1,
\end{align*}
 and, for $M\geq 1$ to be chosen later, consider the decomposition
 $ \mN^{(1)}=\mN^{(1)}_{\leq M} +
 \mN^{(1)}_{> M}$, where $\F \mN^{(1)}_{\leq M}(\tw,\tv)(t,\xi)$ includes the additional restriction to the set $\{ \xi_1 \in \R  : |\O(\xi,\xi_1 ,\xi-\xi_1)|\leq M\}$ and $\mN_{>M}^{(1)} : = \mN^{(1)}-\mN^{(1)}_{\leq M}$.
We recall that due to the extra frequency restriction $\{|\O|\leq M\}$, $\mN^{(1)}_{\leq M}$ can be controlled easily. See \cite[Lemma 3.4]{MP-2023} for a proof.
\begin{lemma}
\label{N-1<M}
If $s\geq 0$ and $0\leq \ta<\frac12,$ then we have
\begin{align*}
\| \mN^{(1)}_{\leq M}(u_1,u_2)  \|_{H^{s+\ta}} \les M \|u_1\|_{H^s} \|u_2\|_{L^2}.
\end{align*}
\end{lemma}

There is no such control on $\mN^{(1)}_{>M}$ and thus we proceed by an integration by parts in time which gains us a factor of $|\O|^{-1}$.
 For this purpose, we introduce the operator
 $\mN^{(1)}_0$ by 
\begin{align*}
 \F \{ \mN^{(1)}_0(u_1,u_2)\}(t,\xi) 
&  =  \int_{\xi=\xi_{12}} 
e^{-it \Omega(\xi,\xi_1,\xi_2) }  \frac{1}{\xi_1} \ind_{\{|\Omega| > M\}}~
 \sigma(\xi,\xi_1,\xi_2)   
\ft{u}_1 (t,\xi_1)\ft{u}_2 (t,\xi_2) d\xi_1.
\end{align*}
Formally, integration by parts in time gives
\begin{align}
\begin{split}
 &\int_0^t \F \{ \mN_{>M}^{(1)}(\tw,\tv) \} (t',\xi) dt'\\
 &= 
\F \{ \mN^{(1)}_0(\tw,\tv) \}(t',\xi)
\Big\vert^{t'=t}_{t'=0} 
 - \int_0^t  
 \int_{\xi=\xi_{12}} 
  \frac{\s e^{-it' \Omega }}{\xi_1} \ind_{\{|\Omega| > M\}}~
 \partial_{t'} \big[  \ft{\tw}(t',\xi_1) \ft{\tv}(t',\xi_2) \big]d\xi_1 dt' \\
 & = \F \{ \mN^{(1)}_0(\tw,\tv) \}(t',\xi)
\Big\vert^{t'=t}_{t'=0}  - \int_0^t  
\F \{ \mN^{(1)}_0(\dtp \tw,\tv)\} (t',\xi)+
\F \{ \mN^{(1)}_0(\tw, \dtp \tv)\}(t',\xi)  dt'.
\end{split} \label{ibp-1-inf}
\end{align}
We can argue as in \cite[subsection 3A]{MP-2023} to justify the first equality in \eqref{ibp-1-inf}; namely, the switching of the $t$ and $\xi_1$ integrals and the use of integration by parts in time. The new difficulty that arises in this paper is in justifying the second equality in \eqref{ibp-1-inf}. Namely, in view of Remark~\ref{RMK:wC1}, we cannot apply the product rule to distribute the time derivative $\partial_{t'}$. To proceed we instead decompose $\ft{\tw}$ into a good and a bad term as in \eqref{dec-w}. For the contribution from the good term, we can apply the product rule and justify the series of equalities in \eqref{ibp-1-inf}. For the bad term, we instead directly prove the identity
\begin{align*}
 &\int_0^t \F \{ \mN_{>M}^{(1)}(B,\tv) \} (t',\xi) dt'\\
  &= \F \{ \mN^{(1)}_0(B,\tv) \}(t',\xi)
\Big\vert^{t'=t}_{t'=0}  - \int_0^t  
\F \{ \mN^{(1)}_0(\dtp B,\tv)\} (t',\xi)+
\F \{ \mN^{(1)}_0(B, \dtp \tv)\}(t',\xi)  dt',
\end{align*}
where $B$ is the bad part of $w$, thus bypassing the first equality in \eqref{ibp-1-inf}. This idea was inspired by \cite[Appendix A]{K19}.  We will also need to follow a similar justification procedure for the second round of integration by parts in Section~\ref{SEC:2nd}.

We now continue with the first step of the normal form reductions.
As in
\cite[subsection 3A]{MP-2023},
we first need to switch the order of integrations in $\xi_1$ and $t'$. We justify this use of Fubini's theorem by using dyadic decompositions. We write $w_{N_1}:=\P_{N_1} w$ and $v_{N_2}:=\P_{N_2}v$, where $\ft w = \sum_{N_1} \ft w_{N_1} $ and $ \ft v = \sum_{N_2} \ft v_{N_2}$
and consider
\begin{align}  
\begin{split}
& \int_0^t \F \{ \mN_{>M}^{(1)}\}(\tw,\tv) (t',\xi) dt' \\
& = -2i \xi \int_0^t \int_{\xi=\xi_{12}}  
e^{-it' \Omega }  
\frac{ \xi_2}{\xi_1} \ind_{\{|\Omega| > M\}}~
 \sigma(\xi,\xi_1,\xi_2)   
 \sum_{N_1} \ft{\tw}_{N_1} (t',\xi_1)
  \sum_{N_2} \ft{\tv}_{N_2} (t',\xi_2) d\xi_1 dt'.  
\end{split} \label{N1M}
\end{align}
The following estimate allows us to justify the switching of the order of the integrals in $\xi_1$ and $t'$ and the summations over $N_1$ and $N_2$: for fixed $\xi\in \R$ and $0<t<T$,   we have
\begin{align}
\begin{split}
  &  \sum_{N_1} \sum_{N_2} \int_0^t \int_{\xi=\xi_{12}}   \Big|
e^{-it' \Omega }  
\frac{\xi_2}{\xi_1} \ind_{\{|\Omega| > M\}}~
 \sigma(\xi,\xi_1,\xi_2)   
  \ft{\tw} (t',\xi_1)
 \ft{\tv} (t',\xi_2)  \Big| d\xi_1 dt' \\
 &  \les T  \sum_{N_1} \sum_{N_2\les N_1}  \frac{N_2}{N_1}\|w_{N_1}\|_{L^\infty_T L^2} \|v_{N_2}\|_{L^\infty_T L^2}\\
& \les  T \|w\|_{L^\infty_T H^s}\|v\|_{L^\infty_T L^2},
\end{split}
\label{FB-1}
\end{align}
for any $s>0$.
Hence, by Fubini's theorem, for fixed $\xi,t$,  we have
\begin{align}
  \begin{split}
\text{LHS} \eqref{N1M} 
& = -2i \xi  \sum_{N_1, N_2}
\intt_{\xi=\xi_{12}} \int_0^t
e^{-it' \Omega }  
\frac{\xi_2}{\xi_1} \ind_{\{|\Omega| > M\}} \s\,
  \ft{\tw}_{N_1} (t',\xi_1)
 \ft{\tv}_{N_2} (t',\xi_2) dt' d\xi_1.   
\end{split} \label{N>M-sum}
\end{align}
We decompose $\tw$ as 
\begin{align}  
 \tw= G+ B,\label{dec-w}
\end{align}
where
\begin{align}
\ft{G}(t,\xi)&:= \hspace{-1mm} -2i  \int_0^t\intt_{\xi=\xi_{12}} 
e^{-it' \Omega } 
 \tfrac{\xi \xi_2}{\xi_1}
 \sigma(\xi,\xi_1,\xi_2)   
\ft{\tw} (t',\xi_1)\ft{\tv} (t',\xi_2) d\xi_1
 dt'  \notag\\
 &\hphantom{XXXX}+ \int_0^t e^{- it' \xi^2} \ft E_1(e^{iF},v)(t',\xi) dt', \label{G}\\
  B(t,x)&:=\int_0^t  e^{it' \dx^2} E_2(e^{iF}, v)(t',x)d t', \label{B}
\end{align}
and $E_1$ and $E_2$ were defined in \eqref{E1E2}.
Then, arguing as in Lemma~\ref{conti}, the first term in the definition of $\ft G(t,\xi)$  belongs to $C^{1}([0,T])$ for each fixed $\xi\in \R$. For the second term in $\ft{G}(t,\xi)$, we note that \eqref{E11} and \eqref{freqE1} imply that $\ft E_1(e^{iF},v)(t',\xi)\in C_{T}L^{\infty}_{\xi}$. Thus,  $\ft G(t,\xi) \in C^{1}([0,T])$ for each fixed $\xi\in \R$.
 Hence, we can apply the integration by parts for the contribution with $G$ and obtain that
\begin{align}
\begin{split}
 &  
 \int_0^t \F \{ \mN_{>M}^{(1)}\}(G,\tv) (t',\xi) dt'\\
& = \sum_{N_1, N_2} 
\bigg\{
\F\{
\mN^{(1)}_0(G_{N_1},\tv_{N_2})\}(t',\xi) \Big\vert^{t'=t}_{t'=0} \\
& \quad \quad 
+  \int_{\xi=\xi_{12}}
\int_0^t  e^{-it' \Omega }  
\frac{1}{\xi_1} \ind_{\{|\Omega| > M\}}~
 \sigma(\xi,\xi_1,\xi_2)   
  \dt \big(\ft{G}_{N_1} (t',\xi_1)
 \ft{\tv}_{N_2} (t',\xi_2) \big) 
 dt' d\xi_1 \bigg\},
\end{split}
 \label{ibp-1'}
 \end{align}
 where $G_{N}:=\P_{N} G.$
Here, using an estimate (for $t'=0,t$) analogous to \eqref{FB-1},
 we see that the first term on the right-hand side of \eqref{ibp-1'}
 is absolutely convergent. Moreover, by Lemma~\ref{conti}, we have $\ft{\tv}(t,\xi), \ft G(t,\xi)\in C^{1}_{T}$ for each fixed $\xi\in \R$, and thus 
\begin{align}
\begin{split}
   \text{LHS}\eqref{ibp-1'}=\F\{ \mN^{(1)}_0(G,\tv)\}(t',\xi)  
\Big\vert^{t'=t}_{t'=0}  &+ 
\sum_{N_1,N_2}
 \int_0^t  
\F \{ \mN^{(1)}_0(\dtp G_{N_1},\tv_{N_2})\}(t',\xi)  dt'
\\
&
+ \sum_{N_1, N_2}  \int_0^t 
\F \{\mN^{(1)}_0(G_{N_1}, \dtp \tv_{N_2})\}(t',\xi)  dt'. 
\end{split}
\label{ibp-1-1}
\end{align}
For the bad contribution in \eqref{N>M-sum}, we need the following identity.

\begin{lemma}\label{IBP-1} 
Let $v\in C_{T}L^2$ be a solution to \eqref{eq-tv}. For $N_1,N_2\in 2^{\mathbb{N}}$, define $ \tv_{N_2} = \P_{N_2}e^{-t\H \dx^2}v(t)$ and $B_{N_1}:=\P_{N_1}B$, where $B\in C_{T} L^2$ is defined in \eqref{B}.
Then, for any $t_1,t_2\in [0,T]$, $M>0$,  
and fixed $\xi\in \R$,
\begin{align} 
  \int_{t_1}^{t_2} 
  \int_{\xi=\xi_{12}} e^{-it \Omega } f(t) d\xi_1 dt 
  =  \bigg[ \int_{\xi=\xi_{12}}
   \frac{e^{-it \Omega} }{-i \Omega } f(t)d\xi_1 
   \bigg]
  \bigg \vert^{t=t_2}_{t=t_1} 
  -\int_{t_1}^{t_2} \int_{\xi=\xi_{12}}  
  \frac{e^{-it \Omega} }{-i \Omega } g(t) d\xi_1 dt, \label{ibp-1-B}
\end{align}
\noi where 
\begin{align*}
f(t)&= -2i  \ind_{\{\xi=\xi_{12}\}}
\tfrac{\xi_2}{\xi_1} \ind_{\{|\Omega| > M\}}~
 \sigma(\xi,\xi_1,\xi_2)   
  \ft{B}_{N_1} (t,\xi_1)
 \ft{\tv}_{N_2} (t,\xi_2), \\
  \begin{split}
    g(t) & = -2i \ind_{\{\xi=\xi_{12}\}} 
\tfrac{\xi_2}{\xi_1} \ind_{\{|\Omega| > M\}}
 \sigma(\xi,\xi_1,\xi_2)    \big[ 
  \ft{B}_{N_1} (t,\xi_1)
  \dt \ft{\tv}_{N_2} (t,\xi_2) \\
  &\hphantom{XXXXXXXXXXXXXXXXXXXXX}+\F\{\P_{N_1} e^{it\dx^2}E_2(e^{iF},v)\}  (t,\xi_1)
 \ft{\tv}_{N_2} (t,\xi_2)   \big],
  \end{split}
\end{align*}
and $\Omega$ and $\s$ are defined in \eqref{Os}. \end{lemma}


\begin{proof} We adapt the 
approximation arguments from \cite[Appendix A]{K19}. 
For $h\neq 0$ such that $ t_1+h, t_2+h \in[0,T]$, we have the following identity
\begin{align}
\begin{split}
&    \int_{t_1+h}^{t_2} \int_{\xi=\xi_{12}}
  \frac{e^{-it \Omega} -e^{-i(t-h) \Omega} }{ih \Omega } f(t) d\xi_1  dt\\
&   =
  \frac{1}{h} \int_{t_2}^{t_2+h}  \int_{\xi=\xi_{12}}
  \frac{e^{-i(t-h) \Omega} }{i \Omega }  f(t) d\xi_1  dt
  - 
  \int_{t_1}^{t_1+h}\int_{\xi=\xi_{12}}
  \frac{e^{-it \Omega} }{i \Omega }  f(t) d\xi_1  dt \\
  & \quad  -\int_{t_1}^{t_2} \int_{\xi=\xi_{12}}
\frac{ e^{-it \Omega}}{i \Omega }
  \frac{ f(t+h)-f(t)  }{ h}  d\xi_1 dt. \\ 
\end{split}\label{IDT}
\end{align}
To prove \eqref{ibp-1-B}, we take limits in \eqref{IDT} as $h\to 0$ by proving the following:
 \begin{align}
\lim_{h\to 0}\int_{t_1+h}^{t_2} \int_{\xi=\xi_{12}}
  \frac{e^{-it \Omega} -e^{-i(t-h) \Omega} }{ih \Omega } f(t) d\xi_1  dt & = 
  \int_{t_1}^{t_2} 
  \int_{\xi=\xi_{12}} e^{-it \Omega } f(t) d\xi_1 dt, \label{lm-1} \\
  \lim_{h\to 0} \bigg[
\int_{t_0}^{t_0+h}  \int_{\xi=\xi_{12}}
  \frac{e^{-i(t-h) \Omega} }{ih \Omega }  f(t) d\xi_1  dt \bigg]
\bigg \vert^{t_0=t_2}_{t_0=t_1} 
    &=  \bigg[
   \int_{\xi=\xi_{12}}
   \frac{e^{-it \Omega} }{-i \Omega } f(t)d\xi_1 \bigg]
  \bigg \vert^{t=t_2}_{t=t_1}, 
   \label{lm-2} \\
  \lim_{h\to0} \int_{t_1}^{t_2} \int_{\xi=\xi_{12}}
\frac{ e^{-it \Omega}}{i \Omega }
  \frac{ f(t+h)-f(t)  }{ h}  d\xi_1 dt &=
   \int_{t_1}^{t_2} \int_{\xi=\xi_{12}}  
  \frac{e^{-it \Omega} }{-i \Omega } g(t) d\xi_1 dt.\label{lm-3}
 \end{align}
First, we verify \eqref{lm-1}.
As
$B_{N_1},v_{N_2}\in C_T L^2$ and $|\xi_1|\ges |\xi_2| $ in the support of $\sigma$, we have 
$f \in C_T L^1_{\xi_1} .$
%
\noi
Thus, by the dominated convergence theorem, we have
\begin{align*}
  \int_{\xi=\xi_{12} }
\frac{ e^{-it \Omega} -e^{-i(t-h) \Omega}  }{ih \Omega } 
  f(t) d\xi_1  \too  \int_{\xi=\xi_{12}} e^{-it \Omega } f(t) d\xi_1, \text{ as  } h\to 0. 
\end{align*} 
 for each fixed $\xi\in \R$ and $t\in [t_1,t_2]$.
\noi
Applying the dominated convergence theorem again we obtain \eqref{lm-1}.
For the limit \eqref{lm-2}, it suffices to show that for fixed $t_0\in [0,T],$  
\begin{align*} 
&\lim_{h\to 0}   \int_{\xi=\xi_{12}} \Big( \int_{t_0}^{t_0+h}
  \frac{e^{-i(t-h) \Omega} }{ih \Omega }  f(t) dt  -
   \frac{e^{-i t_0 \Omega} }{i \Omega }  f(t_0)
   \Big) d\xi_1 =0.
\end{align*}
  Since $f\in C_T L^1_{\xi_1}$, we have
\begin{align*}
 \bigg| \int_{\xi=\xi_{12}}  \int_{t_0}^{t_0+h}  
   \frac{e^{-i t_0 \Omega} }{i  h\Omega }
   \big(  f(t) -f(t_0)  \big)  dt 
d\xi_1 \bigg| &\les  \|f(t) -f(t_0)\|_{L^\infty(( t_0,t_0+h);L^1_{\xi_1})} \too 0, \\
    \bigg|     \int_{\xi=\xi_{12}} 
  \int_{t_0}^{t_0+h} 
  \Big(
  \frac{  e^{-i(t-h) \Omega} -e^{-i t_0 \Omega} }{ih \Omega }  
   \Big)   f(t) 
d\xi_1  dt  \bigg|  &\les h
\|  f\|_{C_TL^1_{\xi_1} } \too 0,
  \end{align*} 
 as $h \to 0$. Thus, the limit \eqref{lm-2} holds true.
Finally, we consider the limit \eqref{lm-3}. For $h\neq 0$, let $D_h$ denote the difference quotient operator, that is,
$D_h f(t)=\tfrac{f(t+h)-f(t)}{h}.$ 
By using the facts $E_2(e^{if},v) \in C_T L^2$ and $\tv\in C_T L^2,$ for fixed $\xi\in \R$ and any $t\in [0,T]$, as well as the identity $D_{h} [ f_1 f_2]  = D_{h}[f_1]f_2 + f_1 D_h [f_2]$
for any two functions $f_1,f_2$,  we have for sufficiently small $h$,
\begin{align*} 
\begin{split}
   \bigg| \intt_{\xi=\xi_{12}}
\frac{ e^{-it \Omega}}{i \Omega }
D_h f(t)  d\xi_1 
  \bigg| 
   &\les  
  \frac{1}{h}\int_t^{t+h} \| E_2(e^{iF},v)(t')\|_{L^2} \|v_{N_2}(t')\|_{L^2} +  \|\dtp v_{N_2}(t')\|_{L^2}  \| B(t')\|_{L^2}dt' 
   \\
   & \les
\|v\|_{C_TL^2} \|E_2(e^{iF},v) \|_{C_TL^2}+ \|\dt \tv_{N_2}\|_{C_T L^2}  \|B\|_{C_TL^2} 
\\
&\les \|v\|_{C_TL^2} \|E_2(e^{iF},v) \|_{C_TL^2}+ N_2^{\frac32} (1 +\dl^{-1}+ \| v \|_{L^\infty_T L^2})^2 \|B\|_{C_T L^2},
 \end{split}
\end{align*}
where we used \eqref{eq-tv}, and the Bernstein inequality in the third inequality.
\noi
Then, by the dominated convergence theorem, it suffices to show that for each fixed $\xi \in \R$, we have
\begin{align}
\lim_{h\to 0}  \int_{\xi=\xi_{12}}  -2i  \tfrac{ e^{-it \Omega}}{i \Omega }
\tfrac{\xi_2}{\xi_1} \ind_{\{|\Omega| > M\}}~
 \sigma
   \big\{ D_h \ft{B}_{N_1} (t,\xi_1) - e^{-it\xi_1^2}\ft E_2(e^{iF},v)(t,\xi_1) \big\}
 \ft{\tv}_{N_2} (t,\xi_2) d\xi_{1} &=0, \label{II-1} \\
 \lim_{h\to 0} \int_{\xi=\xi_{12}}  -2i  \tfrac{ e^{-it \Omega}}{i \Omega }
\tfrac{ \xi_2}{\xi_1} \ind_{\{|\Omega| > M\}}~
 \sigma
 \ft{B}_{N_1} (t,\xi_1)
 \big\{  D_h \ft{\tv}_{N_2}(t,\xi_2) - \dt \ft{\tv}_{N_2}(t,\xi_2)
 \big\}
  d\xi_{1}  &= 0.\label{II-2}
\end{align}
By using that $E_2(e^{iF},v) \in C_T L^2$ and $\tv\in C_T L^2,$ for fixed $\xi\in \R$ and $t\in [0,T]$, we have  
\begin{align*}
  \bigg|  \int_{\xi=\xi_{12}}&  -2i \frac{ e^{-it \Omega}}{i \Omega } 
\frac{ \xi_2}{\xi_1}\ind_{\{|\Omega| > M\}}~
 \sigma(\xi,\xi_1,\xi_2)   
   \big\{ D_h \ft{B}_{N_1} (t,\xi_1) - \ft E_2(e^{iF},v)(t,\xi_1) \big\}
 \ft{\tv}_{N_2} (t,\xi_2) d\xi_{1} \bigg| \\
 & \les 
  \|\tv\|_{C_T L^2}
 \frac1{h} \int_{t}^{t+h}  \|E_2(e^{iF},v)(t',x)-e^{-it \xi_1^2}E_2(e^{iF},v)(t,x)\|_{L^2} dt' \too 0,
\end{align*}
as $h \to 0$, which proves \eqref{II-1}.
For \eqref{II-2}, by Lemma \ref{conti},
 the equation \eqref{eq-tv}, $\tv\in C_TL^2$, and the Bernstein inequality, it holds that 
\begin{align*}
\|
  D_h \ft{\tv}_{N_2}(t,\xi) - \dt \ft{\tv}_{N_2}(t,\xi)
 \|_{L^2_{\xi }} & \leq 
\frac{1}{h} 
 \bigg\|\int_t^{t+h}
\big( \dtp   \ft{\tv}_{N_2}(t',\xi) - \dt \ft{\tv}_{N_2}(t,\xi) \big) dt'
 \bigg\|_{L^{2}_{\xi}} \\
& \les  
 (\dl^{-2}+ N_2^{\frac32}\|v\|_{C_TL^2}) 
 \frac{1}{h}
  \int_t^{t+h} \|v(t)-v(t')\|_{L^2} dt'
  \too 0,
\end{align*}  
as $h\to 0.$
Thus, since $E_2 \in C_T L^2$, we get \eqref{II-2}.
This completes the proof of \eqref{ibp-1-B}.
\end{proof}

Thus, by Lemma~\ref{IBP-1}, we obtain
\begin{align}
\begin{split}
 &
 \int_0^t \F\{ \mN_{>M}^{(1)}\}(B,\tv) (t',\xi) dt'\\
 & = \sum_{N_1, N_2} 
\F\{
\mN^{(1)}_0(B_{N_1},\tv_{N_2})\}(t',\xi) \Big\vert^{t'=t}_{t'=0} \\
& \quad 
+  \sum_{N_1, N_2}   \int_{\xi=\xi_{12}}
\int_0^t  
\frac{\s e^{-it' \Omega }  }{\xi_1} \ind_{\{|\Omega| > M\}} 
   \F\{ 
  \P_{N_1} e^{it' \dx^2}E_2(e^{iF}, v)\}(t',\xi_1)
 \ft{\tv}_{N_2} (t',\xi_2) 
d\xi_1 dt'  \\
 & \quad 
+ \sum_{N_1, N_2}   \int_{\xi=\xi_{12}}
\int_0^t    
\frac{\s e^{-it' \Omega }  }{\xi_1} \ind_{\{|\Omega| > M\}}\ft{B}_{N_1} (t',\xi_1)   \dtp
 \ft{\tv}_{N_2} (t',\xi_2) 
d\xi_1   dt' .
\end{split}
 \label{ibp-1-2}
 \end{align} 
Finally, combining \eqref{ibp-1-1}, \eqref{ibp-1-2}, and using \eqref{dec-w}, we now formally have the identity
\begin{align}
\begin{split}
  \int_0^t \F \{\mN_{>M}^{(1)}\}&(\tw,\tv) (t',\xi) dt'\\
&=\F\{ \mN^{(1)}_0(\tw,\tv)\}(t',\xi)  
\Big\vert^{t'=t}_{t'=0}  + 
\sum_{N_1, N_2} 
 \int_0^t  
\F \{ \mN^{(1)}_0(  \dtp \tw_{N_1}  
  ,\tv_{N_2})\}(t',\xi)  dt'
\\
& \hphantom{XXXXXXXXXXXXX}  
+ \sum_{N_1,N_2}  \int_0^t 
\F ( \mN^{(1)}_0 ( \tw_{N_1} , \dtp \tv_{N_2}))(t',\xi)  dt'.
\end{split}
\label{ibp-1}
\end{align}
The boundary term in \eqref{ibp-1} admits a good estimate. See \cite[Lemma 3.6]{MP-2023}.
\begin{lemma}\label{bd-est-1}
Let $s\geq 0$ and $0\leq \ta<\frac12$. Then, for each fixed time $t\in [0,T]$, 
\begin{align*}
\| \mN^{(1)}_0(u_1,u_2) \|_{H^{s+\ta} }
 \les 
 M^{-\frac18+\frac{\ta}{4}} 
 \|u_1\|_{H^s} \|u_2\|_{L^2}.
\end{align*}
\end{lemma}
 As pointed out in \cite[Remark 3.5]{MP-2023}, the estimates in Lemma \ref{dt-est} are not good enough to ensure that the last two terms on RHS \eqref{ibp-1} are absolutely summable in $N_1$. 
However, we continue to proceed with the normal form reductions for each fixed $N_1$ and $N_2$. Eventually, we will be able to perform the summation in $N_1$ and thus make fully rigorous the justification of the normal form equation we obtain. See the proof of Proposition~\ref{PROP:wtov}.

We now substitute the equations 
\eqref{eq-tw}, and \eqref{eq-tv} for  $\dt \tw$, and $\dt \tv$, respectively, 
into the last two terms of  \eqref{ibp-1}.
First, we consider the substitution of $\dtp \tw$ 
in \eqref{eq-tw}, which is justified by Lemma \ref{bd-est-1} and Lemma \ref{dt-est}. We will rewrite the frequency $N_1$ of $\tw$ as $N_{12}$
and  the frequency $N_2$ of $\tv$ as $N_3$.
For fixed $N_{12},N_3$ and $\xi\in \R$, the term
 $\F \{\mN^{(1)}_0(\dtp \tw_{N_{12} }, \tv_{N_3}  ) \}(t',\xi) dt' $
in \eqref{ibp-1} is equal to
\begin{align}
 & \quad  
 \int_{\xi=\eta+\xi_3} \int_0^t 
  e^{-it' \Omega(\xi,\eta,\xi_3)} 
  \frac{\ind_{\{|\Omega|>M\}}}{\eta}
  \sigma(\xi,\eta,\xi_3) 
  \dtp \ft \tw_{N_{12}}(t',\eta) 
 \ft \tv_{N_3}(t',\xi_3) dt' d\eta  \notag\\
 \begin{split}
   &  = \int_{\xi=\eta+\xi_3}  \int_0^t
  e^{-it' \Omega(\xi,\eta,\xi_3)}
   \frac{\ind_{\{|\Omega|>M\}}}{\eta} \sigma(\xi,\eta,\xi_3)
\psi_{N_{12}}(\eta)
   \\
 & \quad \quad \quad \times
  \bigg( 
\int_{\eta=\xi_{12} } 
 e^{-it'\Omega(\eta,\xi_1,\xi_2)}
  \frac{2\eta \xi_2}{i\xi_1} \sigma(\eta,\xi_1,\xi_2) 
  \ft \tw(\xi_1) \ft \tv(\xi_2) 
   d\xi_1 +\ft E(\eta)  
   \bigg) 
    \ft \tv_{N_3}(t',\xi_3) dt' d\eta 
 \end{split}  \label{s-1} \\
 \begin{split}
 &  = \sum_{N_1, N_2} 
    \int_{\xi=\xi_{123}} \int_0^t 
     e^{-it'\Omega^{(2)}_{1}(\xi,\xi_1,\xi_2,\xi_3) } 
    m^{(2)}_{1}(\xi,\xi_1,\xi_2,\xi_3) 
     \ft \tw_{N_1}(\xi_1) \ft\tv_{N_2}(\xi_2)
    \ft\tv_{N_3}(\xi_3) 
  dt'  d\xi_1d\xi_2\\
    & \quad 
    +  \int_{\xi=\eta+\xi_3}  \int_0^t 
     e^{-it' \Omega(\xi,\eta,\xi_3)}
      \frac{\ind_{\{|\Omega|>M\}}}{\eta}
       \sigma(\xi,\eta,\xi_3) \psi_{N_{12}} (\eta)
    \ft E(\eta) \ft \tv_{N_3}(t',\xi_3)dt'  d\eta ,
 \end{split} \label{s-2}
\end{align}
where $E$ was defined in \eqref{Edef}, $\xi_{123}:=\xi_1+\xi_2+\xi_3$, and
\begin{align}
\begin{split}
\Omega^{(2)}_1(\xi,\xi_1,\xi_2,\xi_3)
&:=\Omega(\xi,\xi_{12},\xi_3)
+\Omega(\xi_{12},\xi_2,\xi_3)=2\xi \xi_3+2\xi_{12}\xi_2,  \\
m^{(2)}_1 (\xi,\xi_1,\xi_2,\xi_3) 
&:=-2i\frac{\xi_2}{\xi_1} \ind_{\{|\Omega|>M\}}
\sigma(\xi,\xi_{12},\xi_3) 
\sigma(\xi_{12},\xi_1,\xi_2) 
\psi_{N_{12} }(\xi_{12})
 \prod_{k=1}^3 \psi_{N_k} (\xi_k).
 \end{split} \label{m12O12}
\end{align}
Here, we discuss the rigour of the last equality \eqref{s-2}.
Since, for fixed $t\in [0,T)$ and $\xi \in \R, $ by Cauchy-Schwarz and Lemma \ref{est-error},  
\begin{align*}
 \int_{\xi=\eta+\xi_3}  \int_0^t  \Big|
      \tfrac{\ind_{\{|\Omega|>M\}}}{\eta}
       \sigma(\xi,\eta,\xi_3) \psi_{N_{12}} (\eta)
    \ft E(\eta) \ft \tv(t',\xi_3) \Big|dt'  d\eta  & \les \tfrac{T}{N_{12}}\|E\|_{C_T L^2} \|v\|_{C_T L^2}\\
 & \les \tfrac{T}{N_{12}}(1+\dl^{-1}+ \|v\|_{C_T L^2})^3 \|v\|_{C_T L^2}, 
\end{align*}
so we can decompose
\begin{align}
\begin{split}
   \eqref{s-1}&  = \int_{\xi=\eta+\xi_3}  \int_0^t
  e^{-it'\Omega(\xi,\eta,\xi_3)}
   \frac{\ind_{
   \{|\Omega|>M\}}}{\eta} \sigma(\xi,\eta,\xi_3)
\psi_{N_{12}}(\eta)
   \\
 & \quad \quad \quad \times
\int_{\eta=\xi_{12} } 
 e^{-it'\Omega(\eta,\xi_1,\xi_2)}
  \frac{2\eta \xi_2}{i\xi_1} \sigma(\eta,\xi_1,\xi_2) 
  \ft \tw(\xi_1) \ft \tv(\xi_2) 
   d\xi_1 
    \ft \tv_{N_3}(t',\xi_3) dt' d\eta 
 \end{split}  \label{s-1'}\\
& \quad  +  \int_{\xi=\eta+\xi_3}  \int_0^t 
     e^{-it' \Omega(\xi,\eta,\xi_3)}
      \frac{\ind_{\{|\Omega|>M\}}}{\eta}
       \sigma(\xi,\eta,\xi_3) \psi_{N_{12}} (\eta)
    \ft E(\eta) \ft \tv_{N_3}(t',\xi_3)dt'  d\eta. \notag
  \end{align}
Next,  to rewrite  \eqref{s-1'} as \eqref{s-2}, dyadically decompose \eqref{s-1'} into dyadic intervals of $\xi_1$ and $\xi_2$, and pull the summations outside of the integrals using Fubini's theorem. This hinges on the fact that the sums and integrals in \eqref{s-2} converge absolutely;
see \cite[(3-30), p.306]{MP-2023}.

For fixed dyadics $N_{12},N_1,N_2,N_3$,
 we define the multilinear operator  $\mN^{(2)}_{1}$ by
\begin{align*}
& \quad \F\{\mN^{(2)}_{1}(\tw,\tv,\tv) \}(t,\xi) =
  \int_{\xi=\xi_{123}}
  e^{-it\Omega^{(2)}_{1}  } 
    m^{(2)}_{1}(\xi,\xi_1,\xi_2,\xi_3)
      \ft \tw(\xi_1) \ft\tv(\xi_2)
\ft\tv(\xi_3)
 d\xi_1d\xi_2.   
\end{align*}

Now, we consider the substitution of
\eqref{eq-tv} for $\dt \tv$ in \eqref{ibp-1}. 
In this case, we rewrite $N_2$ as $N_{23}$.
For fixed $N_1$ and $N_{23}$, we have
\begin{align}
\begin{split}
 &\quad  \mN^{(1)}_0 
 (\tw_{N_1},\dt \tv_{N_{23} } )=    \mN^{(1)}_0 (\tw_{N_1}, \P_{N_{23}} 
      e^{-t\H \dx^2 } \dx (v^2)  )
+  \mN^{(1)}_0 (\tw_{N_1}, \P_{N_{23}} 
      \Qdl \dx \tv  ).
\end{split} \label{dtvsub}
\end{align}
Here, the first term on the right-hand side of \eqref{dtvsub} satisfies
\begin{align}
\begin{split}
 &\quad \F \{\mN^{(1)}_0 (\tw_{N_1},\P_{N_{23}} e^{-t\H \dx^2 } \dx (v^2)) \}(t,\xi) \\
  & =   -i \sum_{N_2} \sum_{N_3}   
 \int_{\xi=\xi_{123}}  
  e^{-i \Omega^{(2)}_2(\xi,\xi_1,\xi_2,\xi_3) }   m^{(2)}_*(\xi,\xi_1,\xi_2,\xi_3) \ft \tw_{N_1}(\xi_1) \ft \tv_{N_2}(\xi_2) \ft\tv_{N_3}(\xi_3) d\xi_1 d\xi_2,
\end{split} \label{dtvfirst}
\end{align}
where 
\begin{align}
\Omega^{(2)}_2 &=\Omega(\xi,\xi_1,\xi_{23})+\Omega(\xi_{23},\xi_2,\xi_3), \label{O2} \\
m^{(2)}_*(\xi,\xi_1,\xi_2,\xi_3)
&=\tfrac{\xi_{23}}{\xi_1} 
\ind_{\{|\Omega(\xi,\xi_1,\xi_{23}) |>M\}} 
\sigma(\xi,\xi_1,\xi_{23})
 \psi_{N_{23}}(\xi_{23}) 
\psi_{N_1} (\xi_1)\psi_{N_2} (\xi_2)\psi_{N_3} (\xi_3).  \notag
\end{align}
The justification of \eqref{dtvfirst} follows similar to the justification of \eqref{s-2} and we thus omit it. See \cite[(3.36)]{MP-2023}.
Following \cite{MP-2023}, we decompose
$\{(\xi,\xi_1,\xi_2,\xi_3) \in \R^4\}=R^{(2)}_{\leq M}\cup R_{>M}^{(2)},$ where
\begin{align*}
R^{(2)}_{\leq M}&= \{|\xi_{12}|\leq 1\} \cup \{ |\Omega^{(2)}_2(\xi,\xi_1,\xi_2,\xi_3)|\leq M \} \quad \text{and} \quad R_{>M}^{(2)} := \R^{4} \backslash R_{\leq M}^{(2)}.
  \end{align*}
We define  
\begin{align}
m_{\leq M}^{(2)} = m^{(2)}_*\ind_{R^{(2)}_{\leq M}} \quad \text{and}\quad m_{2}^{(2)}= m^{(2)}_* 
\ind_{R^{(2)}_{>M}} . \label{m22}
\end{align}
For fixed $N_2,N_3$, we define $\mN_{\leq M}^{(2)}$ and
 $\mN_{2}^{(2)}$ by
 \begin{align*}
\F \mN_{\leq M}^{(2)}(v_1,v_2,v_3)(t,\xi) &=
-i \int_{\xi=\xi_{123}}  
  e^{-i \Omega^{(2)}_2(\xi,\xi_1,\xi_2,\xi_3) } 
    m^{(2)}_{\leq M}(\xi,\xi_1,\xi_2,\xi_3)
     \ft \tw(\xi_1) \ft \tv(\xi_2)
      \ft\tv(\xi_3) d\xi_1 d\xi_2,\\
\F \mN_{2}^{(2)}(v_1,v_2,v_3)(t,\xi) &=
-i \int_{\xi=\xi_{123}}  
  e^{-i \Omega^{(2)}_2(\xi,\xi_1,\xi_2,\xi_3) }   m^{(2)}_{2}(\xi,\xi_1,\xi_2,\xi_3) \ft \tw(\xi_1) \ft \tv(\xi_2) \ft\tv(\xi_3) d\xi_1 d\xi_2.
\end{align*}
Therefore, we may write
\begin{align*}
\mN^{(1)}_0 (\tw_{N_1}, \P_{N_{23}} e^{-t\H \dx^2 } 
\dx (v^2)  )
 = \sum_{N_2} \sum_{N_3} \big\{  \mN^{(2)}_{\leq M}(\tw,\tv,\tv) +  \mN^{(2)}_{2}(\tw,\tv,\tv) \big\} . 
\end{align*}
Note that for simplicity we have omitted the subscripts $N_{j}$ on the right-hand side for clarity.
For the contribution due to $\mN^{(2)}_{\leq M}$, we recall the following lemma, see  \cite[Lemma 3.7]{MP-2023}.
\begin{lemma}\label{N-2<M}
If $s\geq 0$ and $0\leq \ta<\min(s,\frac12)$, then we have
\begin{align*}
\| \mN^{(2)}_{\leq M} (v_1, v_2 ,v_3) \|_{H^{s+\ta} }
\les M^\frac32  \| v_1\|_{H^s} \| v_2\|_{H^s} \| v_3\|_{H^s}.
\end{align*}
\end{lemma}

In summary, after the first step of the normal form reductions we have arrived at
\begin{align}
&\tw(t')\big\vert_{t'=0}^{t'=t}  =
\int_0^t \mN_{\leq M}^{(1)}(\tw,\tv)(t') dt' + \int_0^t \tE(t') dt' +  \mN^{(1)}_{0} (\tw,\tv)(t') \Big|^{t'=t}_{t'=0}   \notag \\
&   -\sum_{N_{12}} \sum_{N_3\les N_{12}} 
\bigg\{      
\sum_{\substack{ N_1, N_2 \\ N_2\les N_1} }
 \int^t_0 \mN^{(2)}_1 (\tw,\tv,\tv)(t')dt'
  + \int_0^t \mN^{(1)}_0(\P_{N_{12}} \tE, \tv_{N_3}) dt'
\bigg\}  \label{wfirst}
\\
&   -\sum_{N_1} \sum_{N_{23}\les N_1} \bigg\{
  \sum_{N_2, N_3}  \int^t_0 \mN^{(2)}_{\leq M} (\tw,\tv,\tv)  +\mN^{(2)}_{2} (\tw,\tv,\tv) dt'   +\int_0^t \mN^{(1)}_0(\tw_{N_1},\P_{N_{23}} \Qdl \dx \tv)  dt'
\bigg\}. \notag
\end{align}
Relative to \cite[(3.42)]{MP-2023}, the only new terms here are the second and fifth (implicitly through \eqref{Edef}), and the last term. We show that the summations in the fifth and the last terms converge absolutely. By Lemma \ref{bd-est-1}, for $t\in[0,T]$, $s> 0$ and $0<\ta_1<\frac12$ and $\ta_2>0$ such that $\ta_2<\min(s, \tfrac 12 -\ta_1)$, we have 
%
\begin{align}
\begin{split}
 &
  \sum_{N_{12}} \sum_{N_3\les N_{12} } 
  \int_0^t
  \| 
  \mN^{(1)}_{0} (\P_{N_{12}} \tilde E, \tv_{N_3} )
  (t') \|_{H^{s+\ta_1} } dt'\\
  & \les T M^{-\frac18+\frac{\ta_1+\ta_2}{4}} 
  \sum_{N_{12}} \sum_{N_3\les N_{12} }
   \jb{N_{12}}^{-\ta_2}
   \|\P_{N_{12}} \wt{E}\|_{L^\infty_T  H^s} 
   \|v\|_{ L^\infty_T  L^2}\\
    & \les T M^{-\frac18+\frac{\ta_1+\ta_2}{4}}
     \|\wt{E}\|_{L^\infty_T  H^s} 
     \|v\|_{L^\infty_T  L^2} \\
     & \les T M^{-\frac18+\frac{\ta_1+\ta_2}{4}} \|v\|_{L^{\infty}_{T}L^2}^{2} \big\{ 1+\|v\|_{L^{\infty}_{T}L^2}+\dl^{-1}(1+\dl^{-1-s})(1+\|v\|_{L^{\infty}_{T}H^{s}})^2\big\},
\end{split}\label{ER-2}
\end{align}
where in the last inequality we used Lemma \ref{est-error} and \eqref{error-3}.
Similarly, by Lemma \ref{bd-est-1} and Lemma \ref{Q-est},
 for $t\in[0,T]$, $s> 0$, we have
\begin{align}
 \sum_{\substack{N_1, N_{23} \\ N_{23}\les N_1}}  \int_0^t  \|\mN^{(1)}_0  (\tw_{N_1},\P_{N_{23}}\Qdl \dx \tv )\|_{H^{s+\ta_1}} dt' 
 \les T M^{-\frac18+\frac{\ta_1+\ta_2}{4}} \dl^{-2} \|w\|_{L^\infty_T H^s}  \|v\|_{L^\infty_T L^2}. \label{dlerror}
\end{align}
It remains to justify that the summations over the terms $\mN^{(1)}_{1}$ and $\mN^{(2)}_2$ are finite and this involves a second round of integration by parts which is the role of the next section.

\subsection{Second step of the normal form reductions}\label{SEC:2nd}


First, we note that due to the frequency restrictions, we have $m_j^{(2)}\in L^{\infty}(\R^4)$ and $|\Omega_{j}^{(2)}|>M$ for each $j=1,2$. This is clear for $j=1$ from \eqref{m12O12} and can be deduced for $j=2$ by further splitting $R^{(2)}_{>M}$ into sub-regions depending on the signs of $\xi_2$ and $\xi_3$. 

We then apply integration by parts to estimate the terms $\mN^{(2)}_{j}$, $j=1,2$. To justify this, we argue as we did in the previous section for the first step. Namely, we need to decompose $\wt{w}$ according to \eqref{dec-w}. As the good term $G$ is continuously differentiable in $t$, we can apply integration by parts and the product rule. For the bad term $B\in C_T L^2$, we use the following lemma, which is analogous to Lemma~\ref{IBP-1}.

\begin{lemma}\label{IBP-2}
Let $v\in C_{T}L^2$ be a solution to \eqref{eq-tv}. For $N_1,N_2,N_3\in 2^{\mathbb{N}}$, define $ \tv_{N_{\l}} = \P_{N_{\l}}e^{-t\H \dx^2}v(t)$ for $\l=2,3$, and $B_{N_1}=\P_{N_1}B$, where $B\in C_{T} L^2$ is defined in \eqref{B}.
Then, for any $t_1,t_2\in [0,T]$, $M>0$, fixed $\xi\in \R$, and for each $j=1,2$, it holds that
\begin{align} 
\begin{split}
  &\quad   \int_{t_1}^{t_2} 
  \int_{\xi=\xi_{123}} e^{-it \Omega^{(2)}_j } f_j(t) d\xi_1d\xi_2 dt  \\
&  =  \bigg[ \int_{\xi=\xi_{12}}
   \frac{e^{-it \Omega^{(2)}_j} }{-i \Omega^{(2)}_j } f_j(t)d\xi_1 
   \bigg]
  \bigg \vert^{t=t_2}_{t=t_1} 
  -\int_{t_1}^{t_2} \int_{\xi=\xi_{12}}  
  \frac{e^{-it \Omega^{(2)}_j} }{-i \Omega^{(2)}_j } g_j(t) d\xi_1d\xi_2 dt, 
\end{split}
\label{ibp-2-B}
\end{align}
\noi where
\begin{align*}
f_j(t)&= 
m_j^{(2)} (\xi,\xi_1,\xi_2,\xi_3)  
  \ft{B}_{N_1} (t,\xi_1)
 \ft{\tv}_{N_2} (t,\xi_2)
  \ft{\tv}_{N_3} (t,\xi_3), \\
    g_j(t) & = m_j^{(2)} (\xi,\xi_1,\xi_2,\xi_3)   
  \ft{B}_{N_1} (t,\xi_1)
  \dt \ft{\tv}_{N_2} (t,\xi_2)  
   \ft{\tv}_{N_3} (t,\xi_3)\\
   & \quad  + m_j^{(2)} (\xi,\xi_1,\xi_2,\xi_3)    
 \ft{B}_{N_1} (t,\xi_1)
 \ft{\tv}_{N_2} (t,\xi_2) 
   \dt \ft{\tv}_{N_3} (t,\xi_3)\\
 &  \quad  +m_j^{(2)} (\xi,\xi_1,\xi_2,\xi_3)   
\F\{\P_{N_1}E_2(e^{iF},v)\}  (t,\xi_1)
 \ft{\tv}_{N_2} (t,\xi_2)   \ft{\tv}_{N_3} (t,\xi_3),
\end{align*}
 $\Omega_{1}^{(2)}$ and $m_{1}^{(2)}$ are defined in \eqref{m12O12}, $\Omega_{1}^{(2)}$ is defined in \eqref{O2}, and $m_{2}^{(2)}$ is defined in \eqref{m22}.
\end{lemma}
\begin{proof}
 We argue as in the proof of Lemma \ref{IBP-2}, taking limits in the analog of \eqref{IDT}. These limits exist and yield \eqref{ibp-2-B} since we have $\ft \tv_N,\dt \ft \tv_N\in C_TL^2_\xi \cap C_T L^1_\xi$ and $E_2(e^{iF},v)\in C_T L^2$ and we use the dominated convergence theorem.

  For the reader's convenience, we show a limit analogous to \eqref{lm-1}: for fixed $\xi$,
  \begin{align}
\int_{t_1+h}^{t_2} \int_{\xi=\xi_{123}}
  \frac{e^{-it \Omega^{(2)}_j} -e^{-i(t-h) \Omega^{(2)}_j} }{ih \Omega^{(2)}_j } f_j(t) d\xi_1  d\xi_2 dt \to 
  \int_{t_1}^{t_2} 
  \int_{\xi=\xi_{12}} e^{-it \Omega^{(2)}_j } f_j(t) d\xi_1 d\xi_2 dt, \label{lm1}
 \end{align} 
  as $h\to 0$.
For each $j=1,2,3$, H\"{o}lder's inequality implies
\begin{align*}
  \int_{\xi=\xi_{123} } |f_j(t)| d\xi_1d\xi_2 & \les  \big( |\ft  B_{N_1} | \ast |\ft  \tv_{N_2} | \ast |\ft \tv_{N_3} |\big)(\xi)
  \les \|B_{N_1}\|_{L^2 }
   \| \tv_{N_2}\|_{L^2 }
    \|\ft  \tv_{N_3}\|_{L^1 }\\
  &  \les N_3 T \|E_2(e^{iF},v)\|_{C_T L^2 }
   \| \tv_{N_2}\|_{C_T L^2 }
    \|\ft  \tv_{N_2}\|_{C_T L^2},
  \end{align*}
  uniformly for $t\in[0,T]$, and thus $f_j\in C_T L^1_{\xi_1,\xi_2}$ which is enough to prove \eqref{lm1}.
It is straightforward then to adapt this idea to prove analogs of the limits \eqref{lm-2} and \eqref{lm-3}.
We omit the details.
\end{proof}

Applying integration by parts  for the contributions with $G$ and Lemma \ref{IBP-2} for those with $B$,  we obtain
\begin{align}
\begin{split}
\int_0^t \mN_j^{(2)}(\tw,\tv,\tv) dt' & = 
\big[ \mN^{(2)}_{j,0}(\tw,\tv,\tv) \big] \big|^{t'=t}_{t'=0} 
- \int_0^t  \mN^{(2)}_{j,0}( \dt \tw,\tv,\tv) dt'\\
& \quad - \int_0^t  \mN^{(2)}_{j,0}( \tw, \dt\tv,\tv) dt'
 -\int_0^t   \mN^{(2)}_{j,0}( \tw,\tv, \dt\tv) dt' ,
\end{split}\label{ibp-2}
\end{align}
where for $j=1,2$, we defined
\begin{align*}
\begin{split}
\F\{\mN_{j,0}^{(2)}(v_1,v_2,v_3)  \}(t,\xi) =\int_{\xi=\xi_{123}} 
e^{-it\Omega^{(2)} _j(\xi,\xi_1,\xi_2,\xi_3)  } 
 \frac{m_j^{(2)} (\xi,\xi_1,\xi_2,\xi_3)}{-i \Omega^{(2)}_j (\xi,\xi_1,\xi_2,\xi_3)} \prod_{k=1}^3  \ft v_{k} (\xi_k) d\xi_1 d\xi_2.
\end{split}
\end{align*} 
Provided that $s>s_0$, we can successfully control all the terms in \eqref{ibp-2} and gain a small power of the largest dyadic frequency to justify the convergence of the sums in \eqref{wfirst}.



\begin{lemma} 
\label{cor}
Let $0<\dl<\infty$ and $s_0<s\leq \frac14 $, where $s_0$ is as in Theorem~\ref{THM:conv}. 
Assume $0<T<1$ and let $v, v^{\dag}\in C_T H^s$ be two
solutions to \eqref{eq-v} with the same initial data. 
Then, there exists $\ta>0$ such that for any $\eps>0$ sufficiently small
\begin{align*}
\big\| \mN^{(2)}_j(\tw,\tv,\tv)  - 
\mN^{(2)}_j(\twda,\tvda,\tvda)  \big\|_{L^1_T H^{s+\ta} } & \les
N_{\max}^{-\eps} 
 (M^{-\frac{1}{8}}
 +T   )
  ( 1 +  \dl^{-1}+
\|v\|_{L^\infty_T H^s}  +\|\vda\|_{L^\infty_T H^s} )^8 \\
& \quad \quad  \times  (  \|w-\wda \|_{L^\infty_T H^s} +  \|v-\vda \|_{L^\infty_T H^s}),
\end{align*} 
for each $j=1,2$, where $\tvda= e^{-t\H \dx^2} \tv$ and $\twda = e^{ it\dx^2 }
\wda$ and $N_{\max}:=\max_{j=1,2,3} N_j$.
 
\end{lemma}
\begin{proof}
We estimate the right-hand side of \eqref{ibp-2}. The boundary term is handled by
\begin{align*}
\max_{j=1,2}\| \mN^{(2)}_{j,0}(v_1,v_2,v_3) \|_{H^{s+\frac12}} 
\les M^{-\frac14+\eps} \|v_1\|_{H^s} 
\|v_2\|_{L^2} \|v_3\|_{L^2},
\end{align*}
for any $s\geq 0$ and $\eps>0$. See \cite[Lemma 3.9]{MP-2023} for a proof. The remaining terms in RHS\eqref{ibp-2} are then estimated as in the proof of \cite[Corollary 3.11]{MP-2023} where the differences $\dt(\tv- \tvda)$ and $\dt(\tw-\twda)$ are now controlled by \eqref{dif-v} and \eqref{dif-w-2}, respectively. Consequently, we have the same restriction $s>s_0$ as in \cite[Corollary 3.11]{MP-2023}. Note that in \cite{MP-2023} the case $j=2$ requires further decomposition but since these terms are unaltered in our setting, we have omitted such detail.
\end{proof}

\begin{proof}[Proof of Proposition \ref{PROP:wtov}]

By the above normal form reductions, we arrive at the normal form equation for the gauged variables $\wt{w}$ given in \eqref{wfirst} with \eqref{ibp-2}.
As each of the summations converge absolutely, when we consider the difference $\wt w$ and $\wt{w}^{\dag}$, we can use telescoping sums to rewrite each of the differences. Then, the estimate \eqref{dif-w} follows from Lemmas~\ref{est-error}, \ref{dif-F}, \ref{N-1<M}, \ref{bd-est-1}, \ref{N-2<M}, \eqref{ER-2}, \eqref{dlerror} and Lemma~\ref{cor}.
\end{proof}

\section{On the periodic setting}\label{SEC:T}

In this section, we consider the ILW on the circle and detail the necessary modifications needed to complete the proof of Theorem~\ref{main-thm} in the setting $\M=\T$.  Namely, establishing unconditional uniqueness for $s_0<s <1/2$, where $s_0$ is as in Theorem~\ref{THM:conv}.

We first discuss the gauge transformation in the periodic case which is slightly different to that on the line in Section~\ref{SEC:GaugeR}. Fix $0<\dl<\infty$.
Let $u\in C([0,T];L^2(\T))$ be a solution \eqref{ILW} for some $T>0$ and let $v$ be its corresponding Galilean transform as in \eqref{Gal}. As $v$ is a distributional solution to \eqref{eq-v} it follows that its mean $\int_\T v(t,x)dx$ is conserved. Upon considering the transformation $v(t,x) \mapsto v(t,x-\tfrac{t}{\pi} \int_{\T} u_0  ) - \tfrac{1}{2\pi} \int_{\T} u_0(x)dx$
we may assume that $v$ is of mean-zero for all times $t$.
We define the primitive $F$ of $v$ via its Fourier coefficients:
\begin{align*}
\ft F (t,0) = 0, \quad \text{and} \quad \ft F(t,n) = \tfrac{1}{in} \ft v(t,n) \quad \text{for} \,\, n\in \Z\setminus\{0\}.
\end{align*}
As $v$ is mean-zero, we have $F=\dx^{-1}v$. Note that $F$ satisfies a version of Lemma \ref{dif-F} trivially due to Bernstein's inequality. Moreover, it solves the equation
\begin{align*}
  \dt F -\H \dx^2 F =v^2-\P_{0}(v^2) +\Qdl v,
\end{align*}
where $\P_0 f := \ft{f}(0)$. In the periodic case, $e^{iF}\in L^2(\T)$ so $\P_{\pm} (e^{iF})$ are well-defined. 
We define the gauge transformed variable 
\begin{align}
w=\dx \P_{+} e^{iF} \label{wperiodic}
\end{align}
 and we see that $w$ satisfies
\begin{align}
\begin{split}
 \dt w- \H \dx^2 w& =-2 \dx \P_{+}(\dx^{-1} w \cdot  \P_{-} \dx v )+ i\dx \P_{+} (e^{iF} \Qdl v)-i\P_{0}(v^2)w .  
\end{split} \label{eq-w-T-m}
\end{align}
Similar to the BO case in \cite{MP-2023}, we cannot assume that the $L^2$-norm of $v$ is conserved in time. Thus, it is beneficial to remove the last term on the right-hand side of \eqref{eq-w-T-m} by using a further gauge transformation 
\begin{align*}
 z(t,x) :=w(t,x) e^{i \g(t)} \quad \text{where} \quad \g(t)=\g(t)[v] : = \int_{0}^{t} \P_{0}(v^2(t'))dt'.
\end{align*}
Then, $z$ solves 
\begin{align*}
\dt z- \H\dx^2 z 
& = -2 \dx \Pbhip(\dx^{-1} z \cdot  \P_{-} \dx v )  +E_{3}[z,v] + e^{i\g}E_{4}[e^{iF}, v],
\end{align*}
where 
\begin{align*}
E_{3}[f,g]: = -2\P_{+}\Pblo( \dx^{-1} f \cdot \P_- \dx g) \quad \text{and} \quad E_{4}[f,g] = i\dx \P_{+}(f \Qdl g). 
\end{align*}
Note that $E_3$ and $e^{i\g} E_4$ satisfy similar estimates to \eqref{error-1} and \eqref{error-2}, respectively.
In particular, for each fixed $n\in \Z$, Cauchy-Schwarz and Lemma~\ref{Q-est} give
\begin{align*}
|\F\{ e^{i\g[f]}\dx \P_{+}[ e^{iF[g]}\Qdl h]\}(n)| \les  |n| \|e^{iF[g]}\|_{L^{2}} \|\Qdl h\|_{L^2}  \les \dl^{-1}\|h\|_{L^2},
\end{align*}
which implies that the maps $t\mapsto \ft z(t,n), \ft{\wt{z}}(t,n)$, where $\wt{z}:=e^{-t\H \dx^2}z$, are continuously differentiable for each fixed $n$. This means that we can apply the product rule in our normal form reductions and thus we do not need to consider a decomposition of the form \eqref{dec-w}.

The refined Strichartz estimates and hence Lemma \ref{dt-est} are slightly different in the periodic case. We give them here to also indicate to indicate the location of the $\dl$-dependence in the estimate which comes from an application of Lemma~\ref{Q-est}. Apart from this, the argument is exactly the same as that in \cite[Lemma 5.1]{MP-2023} so we omit the proof. 
We point out that it is likely possible to adapt the strategy of Kishimoto~\cite{K19} to \eqref{eq-v} and establish unconditional uniqueness for $\frac 16<s<\frac 12$ on $\T$ without having to use the refined Strichartz estimates.

\begin{lemma}\label{LEM:StrT}
Let $0\leq s\leq \frac 14$, $0<\dl<\infty$, $N\in 2^{\N}$, $T>0$, and $2\leq p\leq 4$. If $u$ is a solution to $\dt u-\H\dx^2 u= \dx(u_1 u_2) +\Qdl \dx u_3.$ Then, we have 
\begin{align}
\|\P_N u\|_{L^{p}_{T,x}} \les T^{\frac 1p}N^{\be(s,p)}\big( \|\P_{N}u\|_{L^{\infty}_{T}H^{s}} +\|u_1\|_{L^{\infty}_{T}H^{s}}\|u_2\|_{L^{\infty}_{T}H^{s}} + \dl^{-2}\|u_3\|_{L^{\infty}_{T}L^2} \big), \label{StrichT}
\end{align}
where $\be(s,p) = \big( \tfrac{3}{2}-s\big) \big(\tfrac{1}{4}-\tfrac{1}{2p}\big)-s.$
\end{lemma}

Using \eqref{StrichT}, we can derive analogous results to \eqref{RS-2}, \eqref{dif-v}, and \eqref{dif-w-2}.
All of the normal form considerations are then carried out exactly as in Section~\ref{NF} and the proofs of Theorem~\ref{main-thm} and Corollary~\ref{main-cor} then follow the same arguments. This completes the proof of the unconditional uniqueness for ILW in $C_{T} H^{s}(\T)$ for all $s_0 <s<\frac 14$. Finally, we note that unconditional uniqueness in the range $\frac 14 \leq s <\frac 12$ then follows from the fact that unconditional uniqueness in $C_{t}H^{s}$ implies unconditional uniqueness in $C_{t}H^{s'}$ for any $s' > s$.

\begin{ackno}\rm 
The authors would like to warmly thank Tadahiro Oh for suggesting the problem.
	J.F. and G.L. were supported by the European Research Council
	(grant no. 864138 ``SingStochDispDyn''). 
T.Z.~was supported by the National Natural Science Foundation of China (Grant No. 12101040, 12271051, and 12371239) and by a grant from the China Scholarship Council (CSC). T.Z.~would like to thank the University of Edinburgh for its hospitality where this manuscript was prepared.
The authors would like to thank the anonymous referee for their helpful comments which have greatly improved the presentation of this manuscript.
\end{ackno}

\end{document}